\documentclass{amsart}

\usepackage{amsmath,amssymb,amsthm,mathrsfs,url,color,bbold,esint}
\usepackage[open, openlevel=2, depth=3, atend]{bookmark}
\hypersetup{pdfstartview=XYZ}
\usepackage{cancel}

\newcommand{\C}{\mathbb{C}}
\newcommand{\R}{\mathbb{R}}
\newcommand{\N}{\mathbb{N}}
\newcommand{\Z}{\mathbb{Z}}
\newcommand{\Hh}{\mathbb{H}}

\DeclareMathOperator{\Tr}{Tr}
\DeclareMathOperator{\vol}{vol}
\DeclareMathOperator{\Res}{Res}
\DeclareMathOperator{\Op}{Op}
\DeclareMathOperator{\OTr}{0-Tr}

\DeclareMathOperator{\arcsinh}{arcsinh}

\newtheorem{theorem}{Theorem}
\newtheorem*{theorem*}{Theorem}
\newtheorem{proposition}{Proposition}[section]

\newtheorem{defprop}[proposition]{Definition-Proposition}
\newtheorem{lemma}[proposition]{Lemma}

\title[Counting Resonances]{Weyl laws for manifolds\\ with hyperbolic cusps.}
\author{Yannick Bonthonneau}
\email{yannick.bonthonneau@univ-rennes1.fr}
\address{IRMAR, UMR 6625, 263 avenue du G\'en\'eral Leclerc, CS 74205
35042 RENNES C\'edex.}
\begin{document}

\begin{abstract}
We give Weyl-type estimates on the natural spectral counting function for manifolds with exact hyperbolic cusps. We consider three different cases: without assumption on the compact part, assuming that periodic geodesics form a measure-zero set, and assuming that there are no conjugate points. In each case, we obtain the same type of remainder as in the corresponding case in the context of compact manifolds.

We also investigate the counting of resonances. In particular, we extend results of Selberg to the case of non-constant, negative curvature metrics, under a genericity assumption. 
\end{abstract}

\maketitle

\section{Introduction}

Starting in the 1980's, a series of papers have investigated the Weyl law asymptotics for manifolds with cusps. It is our purpose here to gather the existing statements in the literature, and complete that list with a few results of our own. Since most if not all the arguments used for cusp manifolds are generalizations of arguments developed first for compact manifolds, it seems that to obtain stronger estimates that those presented here, one would need to first improve the asymptotics in the compact case. This is considered to be a difficult problem.

Let us specify what we mean by a manifold with cusps. It is a connected complete Riemannian manifold $(M,g)$ that can be decomposed as the union of a compact part with boundary $M_0$, of dimension $d+1$, and cusps $Z_1, \dots, Z_\kappa$ with
\begin{equation}\label{eq:def-cusp}
(Z_\ell, g) \simeq \left([a_\ell, +\infty)_y \times \left(\R^d/\Lambda_\ell\right)_{\theta} , \frac{dy^2 + d\theta^2}{y^2}\right),
\end{equation}
where each $\Lambda_\ell$ is a lattice in $\R^d$ of covolume $1$ --- this is a normalization condition. Such a manifold has finite volume. We denote the Laplacian $\Delta$ with the analyst's convention that $-\Delta \geq 0$. We let $y_0 = \max_\ell a_\ell$. 

Before we can state any result, we have to introduce some terminology and theorems

\subsection{Spectral theory for manifolds with cusps}

Being a non-compact manifold, the Laplacian cannot have pure point spectrum. It also has continuous spectrum, and we have to explain what we mean by ``counting the spectrum''. The spectral theory of manifolds with cusps was developed in a series of papers \cite{Lax-Phillips-Automorphic-76,CdV-81, CdV-83,Muller-83, Muller-86, Muller-92}. Its conclusions are the following. The Laplacian has both continuous and point spectrum. The point spectrum consists of a discrete set $\sigma_{pp}$ of real numbers $\mu_0=0 < \mu_1 < \dots < \mu_n < \dots$ associated to eigenfunctions $u_0,\dots, u_n,\dots$. This set may be finite or infinite, and each eigenvalue has finite multiplicity. The continuous spectrum has multiplicity $\kappa$, the number of cusps. 

We actually have a decomposition of the spectral measure associated to the continuous spectrum. To each cusp $Z_\ell$ is associated a meromorphic family of generalized eigenfunctions $E_\ell(s)$ called the \emph{Eisenstein functions}, so that
\begin{align*}
(-\Delta - s(d-s))E_\ell(s) &=0, \\
E_\ell(s,x) &= \mathbb{1}\{x\in Z_\ell\} y_\ell^s + \widetilde{E}_\ell(s,x),
\end{align*}
where $\widetilde{E}_\ell(s)\in L^2(M)$ when $\Re s > d/2$. Given $f\in C^\infty_c(M)$, we have the integral representation (see \cite[eq. 7.36]{Muller-83})
\begin{equation}\label{eq:decomposition-L^2}
f = \sum_{\mu_n} \langle f , u_n \rangle u_n + \frac{1}{4\pi} \sum_{\ell=1}^\kappa \int_\R \left\langle f, E_\ell\left( \frac{d}{2} + i \lambda \right) \right\rangle E_\ell\left( \frac{d}{2} + i \lambda \right) d\lambda.
\end{equation}
In cusp $Z_j$, the zeroth Fourier mode of $E_\ell$ takes the form
\[
\delta_{\ell j} y_j^s + \phi_{\ell j}(s) y_j^{d-s},
\]
where $\phi_{\ell j}$ is a meromorphic function. Gathering the $\phi$'s in a matrix $\phi(s)$, and taking the determinant, we obtain the \emph{scattering determinant} $\varphi(s)$. Since the family of Eisenstein function is unique, one can prove that for $s\in \C$
\begin{equation}\label{eq:symmetry-varphi}
\varphi(d-s)\varphi(s) = 1.
\end{equation}
We also have $|\varphi(s)| = 1$ when $\Re s = d/2$. Writing $\varphi(d/2 + i \lambda) = e^{2 i \pi \Psi(\lambda)}$, we obtain a real analytic function $\Psi$ which is called the \emph{scattering phase}.

The resolvent
\begin{equation*}
\mathcal{R}(s) = (-\Delta - s(d-s))^{-1}
\end{equation*}
extends from $\Re s > d/2$ to $\C$ as a meromorphic family of operators $C^\infty_c \to C^\infty$. The poles are in $\{\Re s \leq d/2 \} \cup [d/2, d]$. They form the so-called \emph{resonant set} $\Res(M,g)$. If $d^2/4 + r_n^2$ is a discrete eigenvalue $\mu_n$ and $\Im r_n < 0$, then $d/2 + i r_n$ is in the resonant set. The other elements of the resonant set are the poles of the family of Eisenstein functions, or equivalently the poles of $\varphi$. The poles of $E_i$ that are in $(d/2,d]$ have to be of the form $d/2 + ir_n$.

There will two different meaning to ``counting the spectrum''. On the one hand, we will estimate the mass of the spectral measure by approximating
\begin{equation}\label{eq:def-tilde-N}
\widetilde{N}(\lambda):= N_{pp}(\lambda)- \Psi(\lambda),
\end{equation}
where
\begin{equation}\label{eq:def-N_d}
N_{pp}(\lambda):= \#\{ \mu \in \sigma_{pp}\ | \ \mu \leq d^2/4 + \lambda^2\}.
\end{equation}
On the other hand, we will also be counting the resonant set.

\subsection{Results} 

Our first result is more of a remark:
\begin{theorem}\label{thm:singularities-trace}
Let $M$ be a manifold with $\kappa$ cusps. Then the Fourier transform of $\widetilde{N}'(\lambda)$ is a tempered distribution, whose singular support is contained in the set of algebraic lengths of periodic geodesics on $S^\ast M$.
\end{theorem}
This theorem is just a direct extension to manifolds with cusps of the result of Chazarain \cite{Chazarain-74} for compact manifolds. It hints at the fact that assumptions on the dynamics of the geodesic flow will have a consequence on $\widetilde{N}$. 

Recall that a manifold without conjugate points is a manifold for which the exponential map is a local diffeomorphism everywhere. Manifolds of non-positive curvature do not have conjugate points. However, examples of (compact) manifolds without conjugate points but with patches of positive curvature were exhibited in the papers \cite{Gulliver-75,Ballmann-Brin-Burns-87}; it should be possible to adapt that the constructions therein to create examples on manifolds with cusps. An aperiodic manifold is a manifold for which the set of closed geodesics is a measure-zero set in the unit cotangent bundle. Regarding estimates on $\widetilde{N}(\lambda)$ under different assumptions on the manifold $M$, we get
\begin{theorem}\label{thm:Phase}
Let $(M,g)$ be a manifold with $\kappa$ cusps, of dimension $d+1$. In decreasing order of generality,
\begin{itemize}
	\item Without further assumption on $M$,
\begin{equation}\label{eq:continuous-estimate-general-K}
\widetilde{N}(\lambda) = \frac{\vol B^\ast M}{(2\pi)^{d+1}} \lambda^{d+1} - \frac{\kappa}{\pi} \lambda \log \lambda  + \mathcal{O}( \lambda^d ).
\end{equation}
	\item If one assumes that $M$ is aperiodic then
\begin{equation}\label{eq:continuous-estimate-periodic-zero}
\widetilde{N}(\lambda) = \frac{\vol B^\ast M}{(2\pi)^{d+1}} \lambda^{d+1} - \frac{\kappa}{\pi} \lambda \log \lambda + \frac{\kappa(1-\log 2)}{\pi} \lambda + o( \lambda^d ).
\end{equation}
	\item Finally, if $M$ has no conjugate points, 
\begin{equation}\label{eq:continuous-estimate-negative-K}
\widetilde{N}(\lambda)= \frac{\vol B^\ast M}{(2\pi)^{d+1}} \lambda^{d+1} - \frac{\kappa}{\pi} \lambda \log \lambda + \frac{\kappa(1-\log 2)}{\pi} \lambda + \mathcal{O}\left( \frac{\lambda^d}{\log \lambda} \right).
\end{equation}
\end{itemize}
Observe that the lower order terms are bigger than the remainder only when $d=1$ (surfaces).
\end{theorem}

Selberg proved the last result for \emph{constant} curvature \emph{surfaces} (see (0.2) in \cite{Selberg-2}). Then M\"uller identified the leading term in the variable curvature case in all dimensions (proposition 4.13 in\cite{Muller-86}). For \emph{surfaces} again, Parnovski \cite{Parnovski-95} obtained the first and second estimate (theorem 1.1). Finally the first result was obtained in dimension $n>2$ by Christiansen \cite{Christiansen-01} for a wider class of cuspidal ends; we still include a proof as it is more specific to \emph{hyperbolic cusps}, and also a mere remark on the proof of the third result. The last estimate is the equivalent for cusp manifolds of the main result of B\'erard \cite{Berard-77} for \emph{compact} manifold.

Selberg proved that for constant curvature surfaces, the resonances are contained in a vertical strip near $\{\Re s = 1/2 \}$. The purpose of \cite{Bonthonneau-3} was to extend that result to a set of negatively curved metrics $\mathfrak{G}(M)$, in the following sense. For every metric $g\in \mathfrak{G}(M)$, there is $\eta(g) > d/2$ such that for some $C>0$ and all $\epsilon>0$, we have (see lemma \ref{lemma:resonance-free-zone}):
\begin{equation*}
\# \{ s\in \Res(M,g)\ |\ \Re s < d- \eta(g) - \epsilon ,\ |\Im s| \geq e^{- C \Re s} \} \text{ is finite.}
\end{equation*}
We will explain in section \ref{sec:Counting} how this set of metrics is characterized. Suffice it to say for now that it contains the whole set of negatively curved cusp metrics when there is only one cusp, and it is always a $C^2$-open and $C^\infty$-dense set of metrics. It contains the constant curvature metrics. We conjecture that all negatively curved metrics are in $\mathfrak{G}(M)$.

Each metric $g\in \mathfrak{G}(M)$ comes with two additional constants $a^0_\ast(g)$, $\ell_\ast(g)$, that have an interpretation in terms of \emph{scattered geodesics}, see \cite{Guillemin,Bonthonneau-3}. We obtain
\begin{theorem}\label{thm:Resonances-negative}
Assume $g\in \mathfrak{G}(M)$, and take $\eta > \eta(g)$. Then we have the following estimate. First, for the resonances away from the vertical strip
\begin{equation}\label{eq:counting-out-strip-global}
\#\{ s\in \Res(M,g)\ |\ \Re s < d-\eta,\ |s| \leq \lambda \} = \mathcal{O}(\lambda).
\end{equation}
Actually, we have the more precise estimate: for $0< \tilde{\lambda}<\lambda/2$,
\begin{equation}\label{eq:counting-out-strip-local}
\#\{ s\in \Res(M,g)\ |\ \Re s < d-\eta,\ |s-d/2 - i\lambda| \leq \tilde{\lambda} \} = \mathcal{O}(\tilde{\lambda} + \log \lambda).
\end{equation}
Now, for the resonances in the strip,
\begin{equation}\label{eq:Mangoldt-estimate}
\sum_{\substack{\Re s \geq d-\eta,\\ 0<\Im s \leq \lambda}} d-2\Re s = \frac{\kappa d}{4\pi} \lambda\log \lambda - \frac{\lambda}{2\pi}\left( \frac{\kappa d}{2} + \log |a^0_\ast(g)| - \frac{d}{2}\ell_\ast(g) \right) + \mathcal{O}(\log \lambda).
\end{equation}
And the main estimate,
\begin{equation*}
\begin{split}
\# \{ s& \in \Res(M,g)\ |\ d-\eta \leq \Re s \leq d/2,\ 0\leq \Im s \leq \lambda \} = \\
 &\frac{\vol(B^\ast M)}{(2\pi)^{d+1}} \lambda^{d+1}  -\frac{\kappa}{\pi} \lambda \log \lambda + \left(\kappa(1-\log 2) + \frac{\ell_\ast(g)}{2}\right) \frac{\lambda}{\pi} + \mathcal{O}\left( \frac{\lambda^d}{\log \lambda} \right).
\end{split}
\end{equation*}
\end{theorem}

For constant curvature surfaces, this is due to Selberg \cite{Selberg-2} --- see equations (1.8) and (2.4) therein --- and the $\mathcal{O}$ in equations \eqref{eq:counting-out-strip-global}, \eqref{eq:counting-out-strip-local} can be replaced by $0$. Selberg pointed out that the estimate \eqref{eq:Mangoldt-estimate} is a generalization of the Riemann-Von Mangoldt formula. Still for surfaces, but without assumptions, M\"uller observed \cite{Muller-92} what the leading term in the counting function should be, but there was a gap in the proof. Parnovski closed the gap \cite[Corollary 1.1.a)]{Parnovski-95} and proved a bound on the remainder of size $\mathcal{O}(\lambda^{3/2 + \epsilon})$. The remainder is actually $\mathcal{O}(\lambda^{3/2})$ as was shown in \cite{Bonthonneau-1}. This result can be generalized to any dimension, as the missing ingredient was the estimate \eqref{eq:continuous-estimate-general-K} in higher dimension
\begin{theorem}\label{thm:Resonances-general}
Let $M$ be a manifold with cusps of dimension $d+1>2$. Without further assumption,
\begin{equation*}
\#\{ s\in \Res(M,g)\ |\ |s-d/2| \leq \lambda \} = 2\frac{\vol(B^\ast M)}{(2\pi)^{d+1}} \lambda^{d+1} + \mathcal{O}(\lambda^{d}).
\end{equation*}
We also have the local estimate
\begin{equation*}
\#\{ s \in \Res(M,g)\ |\ |s-d/2 + i\lambda| = \mathcal{O}(1)\} = \mathcal{O}(\lambda^d).
\end{equation*}
Now, assume that $(M,g)$ is aperiodic. Then uniformly in $\epsilon>0$ as $\lambda \to \infty$,
\begin{equation*}
\#\{ s \in \Res(M,g)\ |\ |s-d/2 + i\lambda| \leq \epsilon \} = \mathcal{O}((\epsilon+o(1)) \lambda^d).
\end{equation*}
and
\begin{equation*}
\#\{ s\in \Res(M,g)\ |\ |s-d/2| \leq \lambda \} = 2\frac{\vol(B^\ast M)}{(2\pi)^{d+1}} \lambda^{d+1} + o(\lambda^{d}).
\end{equation*}
\end{theorem}

Section \ref{sec:weyl-phase} will be devoted to the proof of theorem \ref{thm:singularities-trace} and \ref{thm:Phase}. In section \ref{sec:Counting}, we will turn to the proof of theorem \ref{thm:Resonances-negative}. The proof of theorem \ref{thm:Resonances-general} will be found in the last section.

\textbf{Acknowledgement} This work started in Paris, and ended in Montr\'eal, hence funded by both Universit\'e d'Orsay and the CRM. I am happy to thank S. Zelditch for introducing me to the Hadamard parametrix. I am also thankful for suggestions of C. Guillarmou, D. Jakobson, F. Rochon and M. Zworski. I am also very grateful for the patience of the reviewers.

\section{A Weyl asymptotics for the scattering phase}\label{sec:weyl-phase}

As announced, in this section, we will give Weyl-type asymptotics for the scattering phase, for manifolds without conjugate points (the third part of theorem \ref{thm:Phase}, which was the original motivation for this paper). The main idea is to express the quantity we seek to evaluate as some integral involving the trace of the wave group. Then we use the Hadamard parametrix for the wave kernel to obtain an expansion. 

The reason we have a $\log\lambda$ improvement is that in terms of FIO's, the wave group $e^{it \sqrt{-\Delta}}$ does not develop caustics for anytime, so microlocal techniques can be used for times up to $\mathcal{O}(\log \lambda)$, instead of $\mathcal{O}(1)$ in the general case.

In the case of compact manifolds, the trace of the wave group is a well defined distribution. However, for our non-compact manifolds, we have to change the definition of the trace and replace it by a $0$-trace. This is explained in \ref{sec:preliminaries}, where we prepare the stage for the rest of the proof. In the next section \ref{sec:Hadamard}, we recall the construction of the Hadamard parametrix on the universal cover of the manifold. Then, in \ref{sec:contributions-trace}, we evaluate contributions from different elements of the fundamental group. Then comes the conclusion \ref{sec:conclusion-phase}.

As a corollary of the proof, we obtain a Weyl law for general cusp manifolds (without assumption of curvature) by inspecting only the singularity at $0$ of the wave-trace. At last, we sketch proofs for theorem \ref{thm:singularities-trace} and part 2 of theorem \ref{thm:Phase}.

\subsection{The 0-trace}
\label{sec:preliminaries}

We want to define a replacement for the trace of the wave group, and relate it to the spectral quantities. Let $\Pi_0$ be the  $L^2$ orthogonal projector on functions supported in $\{y> y_0\}$, and not depending on $\theta$.

Consider $\tau$ such that $e^\tau > y_0$, and $f$ a Schwartz class function. The operators
\begin{equation*}
(\mathbb{1}-\mathbb{1}_{y>e^{\tau}}\Pi_0) f(-\Delta) \text{ and }\mathbb{1}_{y<e^{\tau}}f(-\Delta)
\end{equation*}
are both trace class because the inclusion $(\mathbb{1}-\mathbb{1}_{y>e^{\tau}}\Pi_0) H^N(M)\hookrightarrow L^2(M)$ is trace class for $N>d+1$ --- see proposition 2.3 combined with proposition 1.23 in \cite{Bonthonneau-2} for example.

\begin{lemma}[Trace formula]
Let $\psi \in C^\infty_c(\R)$ be real and even. Then
\begin{equation}\label{eq:0-trace-psi}
\begin{split}
\lim_{\tau\to+\infty}& \left[\Tr \left\{ (\mathbb{1}-\mathbb{1}_{y>e^{\tau}}\Pi_0) \widehat{\psi}\left( \sqrt{-\Delta - \frac{d^2}{4}}\right)\right\} - \kappa \tau \psi(0)\right] \\
&= \sum_{\mu\in \sigma_{pp}} \widehat{\psi}(\sqrt{\mu - d^2/4}) - \frac{1}{2} \int \Psi'(\lambda) \widehat{\psi}(\lambda)d\lambda + \frac{1}{4} \widehat{\psi}(0) \Tr \phi(d/2).
\end{split}
\end{equation}
The quantity in the LHS is denoted by $\OTr \widehat{\psi}\left( \sqrt{-\Delta - \frac{d^2}{4}}\right)$. The limit also exists when one replaces $\mathbb{1}_{y>e^{\tau}}\Pi_0$ by $\mathbb{1}_{y> e^\tau}$, and it is the same.
\end{lemma}

This is a statement similar to equation (2.2) in \cite{Muller-92}. For a more general exposition of the concept of $0$-trace we refer to \cite{Guillope-Zworski}.
\begin{proof}
The first observation is that the operator $\widehat{\psi}(\sqrt{-\Delta- d^2/4})$ is well defined by functional calculus, because $\widehat{\psi}$ is an even entire function. From \eqref{eq:decomposition-L^2}, we find that 
\begin{equation}\label{eq:trace-cutoff}
\begin{split}
\Tr\left\{ \mathbb{1}_{y<e^{\tau}} \widehat{\psi}\left( \sqrt{-\Delta - \frac{d^2}{4}}\right)\right\}& = \sum_{\mu} \widehat{\psi}(\sqrt{ \mu - d^2/4}) \int_{y\leq e^\tau} |u_\mu|^2\\
+ \frac{1}{4\pi} &\sum_{\ell} \int \widehat{\psi}(\lambda) \int_{y< e^\tau} |E_\ell(d/2 + i\lambda,x) |^2 dx d\lambda,
\end{split}
\end{equation}
provided the RHS makes sense, by linearity of the trace. Since we have a Weyl upper bound on the number of eigenvalues (see theorems 3.11 and 4.1 in \cite{Muller-86}), the first term in the RHS is a certainly a converging sum, which converges as $\tau \to + \infty$ to 
\begin{equation*}
\sum_{\mu} \widehat{\psi}(\sqrt{ \mu - d^2/4}).
\end{equation*}
Since this also holds replacing $1- \mathbb{1}_{y>e^\tau}$ by $1- \mathbb{1}_{y> e^\tau} \Pi_0$, we can focus our attention on the contribution from the continuous spectrum. Now, we need to use the Maass-Selberg relations:
\begin{equation}\label{eq:Maass-Selberg-bis}
\begin{split}
\int_M &(\mathbb{1}-\mathbb{1}_{y>e^\tau}\Pi_0) \sum_\ell |E_\ell|^2  \\
&= 2 \kappa \tau - \frac{\varphi'}{\varphi}\left( \frac{d}{2} + i \lambda \right) + \Tr\frac{ e^{2i \tau \lambda} \phi^\ast(\frac{d}{2} + i\lambda) - e^{-2i \tau \lambda} \phi(\frac{d}{2} + i\lambda)}{2i\lambda}\ .
\end{split}
\end{equation}
This the content of \eqref{eq:Maass-Selberg-axis} from the appendix. Here the first term is a constant in the $\lambda$ variable, the second is $-2\pi \Psi'$ and the last one is a continuous bounded function. However, using the fact that $\Psi(\lambda) = \mathcal{O}(\lambda^{d+1})$ and equation \eqref{eq:decomp-increasing-phase} also from the appendix, we deduce that 
\begin{equation*}
\int_0^\lambda |\Psi'(\lambda')|d\lambda' = \mathcal{O}(\lambda^{d+1}).
\end{equation*}
It follows that the contribution from the continuous spectrum in the RHS of equation \eqref{eq:trace-cutoff} is a converging integral. The difference between the case $\mathbb{1}_{y<e^\tau}$ and $\mathbb{1}-\mathbb{1}_{y> e^\tau} \Pi_0$ is 
\begin{equation*}
\frac{1}{4\pi}\sum_\ell \int \int_{y>e^\tau} (1-\Pi_0)|E_\ell(d/2 + i\lambda,x)|^2 \widehat{\psi}(\lambda) dx d\lambda .
\end{equation*}
When $\tau \to +\infty$, we use dominated convergence to show that it goes to $0$.

Inspecting the several identities, the proof of the lemma will be complete if we can show that 
\begin{equation*}
\begin{split}
I(\tau) := \frac{1}{4\pi} \int & \widehat{\psi}(\lambda) \frac{e^{2i\tau \lambda} \Tr \phi^\ast(d/2 + i\lambda) - e^{-2 i \tau \lambda} \Tr \phi(d/2 + i\lambda)}{2i\lambda} d\lambda \\
	& \underset{\tau\to +\infty}{\longrightarrow} \frac{1}{4} \widehat{\psi}(0) \Tr \phi(d/2).
\end{split}
\end{equation*}
We compute
\begin{equation*}
\partial_\tau I = \frac{1}{4\pi} \int \widehat{\psi}(\lambda)\times\left( e^{2i\tau\lambda} \Tr \phi^\ast(d/2+i\lambda) + e^{-2i\tau \lambda}\Tr \phi(d/2 + i\lambda)\right)d\lambda.
\end{equation*}
Since $\phi^\ast(d/2 + i\lambda) = \phi(d/2 - i\lambda)$,
\begin{equation*}
\partial_\tau I = \frac{1}{2\pi} \psi \ast \widehat{\Tr \phi} (2 \tau).
\end{equation*}
(The Fourier transform is taken along the $\lambda$ parameter). In particular, this has fast decay when $\tau \to \pm \infty$, so $I$ has limits for $\tau \to \pm \infty$. But, as $\psi$ is real and even, $\widehat{\psi}$ also is. Hence, $I$ is real and $I(-\tau) = - \overline{I(\tau)}$, so 
\begin{equation*}
2 I(+\infty) = \int \partial_\tau I = \frac{1}{2} \widehat{\psi}(0) \Tr \phi(d/2).
\end{equation*}
\end{proof}

Let $K(t,x,x')$ be the Schwartz kernel of $\cos t \sqrt{-\Delta - d^2/4}$. From the considerations in the second half of p. 44 of \cite{Duistermaat-Guillemin-75} we deduce that the restriction to the diagonal $K(t,x,x)\in C^\infty(M_x, \mathcal{D}'(\R)_t)$. So we can reformulate the lemma above to
\begin{equation}\label{eq:Trace-formula-2}
\begin{split}
\sum_\lambda \widehat{\psi}\left(\sqrt{\lambda - \frac{d^2}{4}}\right)& - \frac{1}{2} \int \Psi'(\lambda) \widehat{\psi}(\lambda)d\lambda+ \frac{1}{4} \widehat{\psi}(0) \Tr \phi\left(\frac{d}{2}\right) = \\
&\lim_{\tau \to \infty} \int_{y \leq e^\tau} \int_\R \psi(t) K(t,x,x)dt dx  -\kappa \tau \psi(0).
\end{split}
\end{equation}
The density $dx$ here (and afterwards) is understood as the Riemannian volume form. Beware that in the cusps, it takes the form $y^{-d-1}d y d\theta$.

\subsection{The Hadamard parametrix}
\label{sec:Hadamard}

From now until section \ref{sec:no-assumption-curvature}, we assume that there are no conjugate points. We will build a parametrix for $K(t,x,x')$, first on the universal cover, and then we will come back to the manifold by summing over the fundamental group.

\subsubsection{Building the approximation on the universal cover}

B\'erard gave an elegant exposition of the Hadamard parametrix in \cite{Berard-77}. It was our intent to use it directly. However, doing this, a remainder appeared, and it was not obvious that it was trace class. To avoid such a discussion, we found it was simpler to use a version of the Hadamard parametrix tweaked to our needs. Instead of being modelled on the Euclidean space, it is modelled on the hyperbolic space. For lack of a reference, we go through its construction, which is very similar to the original parametrix.

We have to introduce a bit of notation. We follow B\'erard (or, if you will, \cite{Berger-Gauduchon-Mazet}). Let $\widetilde{K}(t,x,x')$ be the kernel of $\cos t \sqrt{-\Delta -d^2/4}$ on $\widetilde{M}$. If $x,x' \in \widetilde{M}$, $r=d(x,x')$ is the riemannian distance between them. Since the manifold has no conjugate points, the exponential map is a local diffeomorphism, and so it is bijective from any $T_x M$ to $\widetilde{M}$. In particular, we have exactly one geodesic between any two points on $\widetilde{M}$. We let 
\begin{equation*}
\Theta(x,x') = \det T_{\exp^{-1}x'} \exp_x.
\end{equation*}
(This is the Jacobian of the exponential map.) For fixed $x$, $\partial/\partial r$ is the unit vector field along geodesics from $x$, and $\Theta'= \partial_r \Theta$. From the proposition G.V.3 in \cite{Berger-Gauduchon-Mazet}, we have
\begin{equation*}
\nabla\cdot\left(\frac{\partial}{\partial r}\right)= \Delta r = -\left( \frac{\Theta'}{\Theta} + \frac{d}{r}\right). 
\end{equation*}
From this identity, one may compute $\Theta$ in the case of the real hyperbolic space $\Hh^{d+1}$. It is only a function of $r$, and we denote it by $\Theta_0(r)$. One finds that
\begin{equation*}
\frac{\Theta'_0}{\Theta_0} = d \left(\coth r - \frac{1}{r}\right),
\end{equation*}
so that coming back to $\widetilde{M}$,
\begin{equation}\label{eq:div-d/dr}
\nabla\cdot\left(\frac{\partial}{\partial r}\right) = -\left( \frac{\Theta'}{\Theta} - \frac{\Theta'_0}{\Theta_0} + \frac{d \cosh r}{\sinh r}\right). 
\end{equation}

We define a family of tempered distributions on $\R$ in the following way. If $s\in \R$, $s_+$ is its positive part. For $\alpha \in \C$,
\begin{equation*}
M_\alpha(s) = \begin{cases} \frac{s_+^\alpha}{\Gamma(\alpha+1)} \text{ if $\Re \alpha > - 1$} \\ M_{\alpha-1}(s) = M_\alpha'(s). \end{cases}
\end{equation*}
In particular, we have $s M_{\alpha - 1}(s) = \alpha M_{\alpha}(s)$. 

The wave equation propagates at speed $1$. That is to say that $\widetilde{K}(t,x,x')$ is supported for $d(x,x')\leq t$. Additionally, the singular part of the kernel is supported \emph{exactly} on $\{d(x,x') = t\}$ (see theorem 6.1 in \cite{Taylor-1}).

The point of the Hadamard parametrix is to expand the kernel of the wave operator in powers of $(t^2 - r^2)_+$. However, the first terms in the development have to be negative powers. To define them as distributions, we have to interpret $(t^2 -  r^2)_+^{-n}$ as $M_{-n}(t^2 -  r ^2)$, if $n \geq 0$. To be more precise, since $(r,t)\mapsto t^2 - r^2$ is not a submersion at $0$, and $M_{\alpha}$ has non-trivial wavefront set at $0$, $M_{\alpha}(t^2 - r^2)$ is not a well defined distribution; it is $|t| M_\alpha(t^2- r^2)$ that is well defined. Indeed, we set
\begin{equation*}
\int M_\alpha(t^2 - r^2)|t| f(t,x)dt dr := \int \frac{f(\sqrt{u},x) + f(-\sqrt{u},x)}{2}M_\alpha(u - r^2) du dr.
\end{equation*}
The key here is that when $f$ is a smooth function, $(f(\sqrt{u},x) + f(-\sqrt{u},x))/2$ also is. In the case of the real hyperbolic space, it is convenient to replace $t^2 -  r ^2$ by $\cosh t - \cosh  r $ --- they have a similar behaviour around $(0,0)$. That is, we are looking for an expression of the kernel in the form
\begin{equation*}
\widetilde{K}(t,x,x') = C_0 \sum_{k} \left( \frac{-1}{2}\right)^k f_k(x,x') \sinh |t| M_k(\cosh t - \cosh  r ),
\end{equation*}
where we are summing over $k_0 + \N$, $k_0 \in \R$. Obviously such an expansion cannot be converging, but let us do a formal computation. Let $\square = \partial_t^2 -\Delta_{x'} -d^2/4$, and $\square_0=\partial_t^2 - \Delta_{x'}$. For $t>0$,
\begin{equation*}
\begin{split}
\square \left\{ f_k \sinh t M_k \right\}& = \{(-\Delta - d^2/4)f_k \}\sinh t M_k \\
					&{}+ f_k \sinh t \square_0 M_k \\
					&{}+ f_k \sinh t M_k\\
					&{}+ 2 f_k \cosh t \partial_t M_k \\
					&{}- 2\sinh t \nabla f_k . \nabla M_k.
\end{split}
\end{equation*}
We compute each term in the RHS
\begin{align*}
\nabla M_k 		& = - \sinh r M_{k-1} \partial_r, \\
\partial_t M_k 	& = \sinh t M_{k-1},\\
\square_0 M_k 	& = \left[k(\cosh t + \cosh r) + \sinh r \left( \frac{\Theta'}{\Theta} - \frac{\Theta'_0}{\Theta_0}\right) + d \cosh r \right]M_{k-1}.
\end{align*}
Since we want to solve the equation $\square \widetilde{K} = 0$, we get the equation
\begin{equation*}
\begin{split}
& 0  = C_0 \sum \left(\frac{-1}{2}\right)^k \sinh|t|M_k \times \left[ \Big(-\Delta - d^2/4 + 1 + k(k+2)\Big)f_k \right.\\
	&\left. -\frac{1}{2}\left\{ f_{k+1}\left(2(k+2 + \frac{d}{2})\cosh r + \sinh r \left(\frac{\Theta'}{\Theta} - \frac{\Theta'_0}{\Theta_0}\right) \right) + 2 \sinh r\partial_r f_{k+1} \right\}\right].
\end{split}
\end{equation*}
We deduce that the coefficients must satisfy the relations
\begin{equation*}
\begin{split}
f_{k}\left( \left(k+1+\frac{d}{2}\right)\cosh r +\frac{1}{2}\sinh r\left( \frac{\Theta'}{\Theta}-  \frac{\Theta_0'}{\Theta_0}\right) \right)& + \sinh r\partial_r f_{k} =\\
& (-\Delta  + k^2 - d^2/4)f_{k-1}.
\end{split}
\end{equation*}
This is a nice family of transport equations. Now, we need to determine $k_0$. Since we require that the limit of $\widetilde{K}$ when $t \to 0$ is $\delta(x,x')$, by a dimensional analysis, the sum has to start at $k_0 = -d/2 - 1$. The corresponding $f_{k_0}$ is $(\Theta_0/\Theta)^{1/2}$, up to a constant. We let $u_k = f_{k_0 + k}$ for $k\in \N$.

The solution to the system is:
\begin{align*}
u_0				&= \sqrt{\frac{\Theta_0}{\Theta}} \\
\intertext{and for $k>0$} u_{k}			&= \sqrt{\frac{\Theta_0}{\Theta}} \int_0^r \frac{\sinh(s)^{k-1}}{\sinh (r)^k} \sqrt{\frac{\Theta}{\Theta_0}(s)} (-\Delta  + (k-1-d/2)^2- d^2/4) u_{k-1}(s) ds,
\end{align*}
where $s$ parametrises the geodesic between $x$ and $x'$, travelled at speed $1$. When the curvature around the geodesic from $x$ to $x'$ is constant equal to $-1$, $u_k$ vanishes at $(x,x')$ for $k \geq 1$, and $u_0 = 1$. 

\begin{defprop}\label{defprop:C^n-norms}
Let $(M',g')$ be a Riemannian manifold. For $k\in \N$, and $f\in C^k(M')$, we let
\begin{equation*}
\| f \|_{\mathscr{C}^k(M')} := \sup_{n\leq k} \| \nabla^n f \|_{L^\infty(M')}.
\end{equation*}
If $M' = Z$ is a cusp, $\|f\|_{\mathscr{C}^k(Z)}$ is equivalent to 
\begin{equation*}
\| f \|_{\mathscr{C}^k(Z)}' := \sup_{\alpha + |\beta| \leq k} \| (y\partial_y)^\alpha(y\partial_\theta)^\beta f\|_{L^\infty(Z)}.
\end{equation*}
\end{defprop}

For a proof, we refer to appendix A in \cite{Bonthonneau-2}. The following lemma is crucial.
\begin{lemma}\label{lemma:eta-1ound-u_k}
For all $k,l\geq 0$, $\| u_k \|_{\mathscr{C}^l(d(x,x')<r)} = \mathcal{O}(1) e^{\mathcal{O}( r )}$.
\end{lemma}

\begin{proof}
The lemma is based on the fact that solutions to linear ODE's with bounded coefficients grow at most exponentially fast.

Given a function $f$ on $\widetilde{M}\times\widetilde{M}$, we will say that it is \emph{tame} if it satisfies a similar set of estimates: $\|f\|_{\mathscr{C}^l(d(x,x')<r)} = \mathcal{O}(1) e^{\mathcal{O}( r )}$ for all $l\geq 0$. The set of tame functions is an algebra. It is also stable by $\Delta_{x'}$ --- since $\Delta = \Tr \nabla^2$.
\begin{lemma}\label{lemma:r-is-tame}
The distance function $r^2=d(x,x')^2$ is tame.
\end{lemma}

\begin{proof}
This follows from the fact that the curvature tensor $\mathscr{R}$ and all its covariant derivatives are bounded on $\widetilde{M}$. 

Consider the matrix Jacobi field $\mathbb{A}_x$ along the geodesic between $x$ and $x'$, that vanishes at $x$, and such that $\mathbb{A}_x'(x)=\mathbb{1}$. Then $\nabla_{x'} \nabla_{x'} r = \mathbb{A}_x'\mathbb{A}_x^{-1}(x')$. According to lemma 2.8 in \cite{Eberlein-73}, this is bounded by $k\coth k r$ for $k > \| \mathscr{R} \|_{L^\infty}^{1/2}$. Now, consider the matrix Jacobi field $\mathbb{B}$ such that $\mathbb{B}(x')=0$ and $\mathbb{B}(x)=\mathbb{1}$. Then $\nabla_x \nabla_{x'} r = \mathbb{B}'(x')$. This is also $\mathbb{A}_{x'}^{-1}(x)$. As one can see from the arguments in page 1299 in \cite{Bonthonneau-6}, this is $\mathcal{O}(\langle 1/r \rangle)$. 

To produce the desired estimates for the derivatives of order $n>2$ it suffices to control $\nabla^{n-2}_{x,x'} (\mathbb{A}_{x}(x'))$ and $\nabla^{n-2}_{x,x'} (\mathbb{A}_{x}'(x'))$. It boils down to computing the variation of solutions to a second order linear ODE when initial conditions are constant and the coefficients of the equation vary. The equation for $X=\nabla^{n-2}_{x,x'} (\mathbb{A}_{x}(x'))$ can be put in the form
\[
X'' + K X = P(\mathbb{A}, \nabla \mathbb{A},\dots,\nabla^{n-3}\mathbb{A}, \mathscr{R}, \dots, \nabla^{n-2}\mathscr{R}),
\]
where $K$ involves the curvature tensor, and $P$ is a polynomial expression with universal constant coefficients. Using usual techniques of linear ODE's (variation of parameters), we can proceed by induction and finish the proof.
\end{proof}

As a consequence, if $f\in C^\infty(\R)$ is even and $\| f \|_{C^n(|x|\leq r)} = \mathcal{O}(1) e^{\mathcal{O}( r )}$, then $f(r)$ is tame. In particular, $\Theta_0$ is tame. It was the main technical result in \cite{Bonthonneau-6} (lemma 3) that 
\[
\frac{1}{\sqrt{\Theta}} = \mathcal{O}( 1 + r^{d/2}).
\]
Since $\Theta$ can be expressed as $r^{-d}\det \mathbb{A}_x$ where $\mathbb{A}_x$ is the matrix-Jacobi field along the geodesic between $x$ and $x'$ introduced above, we deduce that very much like the case of $r$, the covariant derivatives of $\Theta$ will be solutions of linear second order ODE's with bounded coefficients, and tame forcing terms. As a consequence $\Theta$ is tame, and so is $u_0$. 

To conclude, proceed by induction and assume that $u_k$ is tame for some $k\geq 0$. Then inspect the formula for $u_{k+1}$. It is an integral expression involving only tame functions. Combine the algebraic stability of tame functions with differentiation under the integral sign to deduce that $u_{k+1}$ is tame.
\end{proof}

As the sum does not converge, we define the cut-off sums
\begin{equation*}
\widetilde{K}_N(t,x,x') := C_0 \sum_{k=0}^N \left( \frac{-1}{2}\right)^k u_k(x,x') \sinh|t| M_{k-(d+2)/2}(\cosh t - \cosh  r ).
\end{equation*}

\begin{lemma}
There is a constant $C_0$ depending only on the dimension, so that for $N\geq 0$, $x\in \widetilde{M}$, for all $\psi \in C^\infty_c(\widetilde{M})$, as $t\to 0$,
\begin{equation*}
\int \widetilde{K}_N(t,x,x')\psi(x')dx' \to \psi(x) \text{ and } \int \partial_t\widetilde{K}_N(t,x,x')\psi(x')dx' \to 0.
\end{equation*}
Additionally, for $N > (d+2)/2$,
\begin{equation}\label{eq:remainder-1}
\begin{split}
\square \widetilde{K}_N = C_0 &\left( \frac{-1}{2} \right)^N \left[(-\Delta +(N-d/2)^2 - d^2/4)u_N  \right] \\
&\times \sinh|t| \left(\cosh t - \cosh  r \right)^{N-(d+2)/2}_+.
\end{split}
\end{equation}
\end{lemma}

The constant $C_0$ depending only on the dimension, we can compute its value in the case of constant $-1$ curvature. But that value can be found (for example) in \cite{Bunke-Olbrich-94}, page 360, in proposition 2.1. We get
\begin{equation*}
C_0 = \frac{1}{2}\frac{1}{(2 \pi)^{d/2}}.
\end{equation*}
The computations needed to check the limits at $t \to 0$ can be found in \cite{Berard-77} --- see proposition 27 therein. \qed

\subsubsection{Summing over the fundamental group}

Now, we denote by $\Gamma$ the fundamental group of $M$, and define for $N\geq 0$,
\begin{equation*}
K_{N}(t,x,x') = \sum_{\gamma \in \Gamma} \widetilde{K}_N(t,x,\gamma x').
\end{equation*}
For each $x,x'$, the number of non-vanishing terms is finite, so this is well defined. Since $K_N(t,\cdot)$ is bi-invariant by $\Gamma$, it defines a kernel on $M$, that we still denote by $K_{N}(t,\cdot)$. The associated operator on $C^\infty_c(M)$ is denoted by $A_{N}(t)$: for $f\in C^\infty_c(M)$,
\begin{equation*}
A_N(t) f (x) = \int_{M} K_N(t,x,x') f(x')dx'
\end{equation*}

\begin{lemma}\label{lemma:remainder-Hadamard}
For $N > 10 d $, $R_N(t) := A_N(t) - \cos t \sqrt{-\Delta - d^2/4}$ is a continuous family of trace class operators, with
\begin{equation}\label{eq:trace-estimate-R_N}
\Tr R_N(t) = \mathcal{O}(t) e^{\mathcal{O}(|t|)}.
\end{equation}
\end{lemma}

\begin{proof}
First, let $V_N(t) = \square R_N(t)$. Then we have
\begin{equation*}
R_N(t) = \int_0^t \frac{ \sin \left[ (t-s) \sqrt{-\Delta - d^2/4}\right]}{\sqrt{-\Delta - d^2/4}} V_N(s) ds.
\end{equation*}
Since $\sin (s \sqrt{-\Delta - d^2/4}) / \sqrt{-\Delta - d^2/4}$ is bounded with norm $\mathcal{O}(1) e^{\mathcal{O}(|s|)}$, it suffices to prove an estimate similar to \eqref{eq:trace-estimate-R_N} for $V_N(t)$. We need

\begin{lemma}\label{lemma:trace-norm}
Let $L(x,x')$ be the kernel of some operator on $L^2(M)$. Assume that
\begin{equation*}
\| L \|' := \sum_{|\alpha| \leq 2d + 3} \| y(x)^{d/2} y(x')^{d/2}\nabla^\alpha_{x,x'} L \|_{L^1(M \times M)} < \infty.
\end{equation*}
Then the corresponding operator is trace class, and its trace norm is controlled by $\| L \|'$. (Here, $y$ is a height function, corresponding with the usual function in the cusps, and being some positive constant in the compact part $M_0$).
\end{lemma}

We will give the proof of this fact later. For now,
\begin{equation*}
\|V_N\|'\leq \sum_{|\alpha |\leq 2d+3} \sum_{\gamma \in \Gamma} \int_{D \times D} |\nabla^\alpha_{x,x'}\square \widetilde{K}_N(t,x,\gamma x')| y(x)^{d/2} y(x')^{d/2} dx dx'.
\end{equation*}
where $D$ is a fundamental domain for the action of $\Gamma$ on $\widetilde{M}$. We can rewrite this as
\begin{equation*}
\sum_{|\alpha |\leq 2d+3} \int_{D \times \widetilde{M}} |\nabla^\alpha_{x,x'}\square \widetilde{K}_N(t, x, x')| y(x)^{d/2} y(x')^{d/2} dx dx'.
\end{equation*}
Since the curvature is constant $-1$ in the cusp, $\square \widetilde{K}_N$ vanishes for $y(x),y(x')$ larger than $y_0 e^{|t|}$. Hence, $\|V_N\|'$ is bounded by 
\begin{equation*}
\mathcal{O}(1) e^{|t|.d} \int_D \sum_{|\alpha |\leq 2d+3} \int_{\widetilde{M}} |\nabla^\alpha_{x,x'}\square \widetilde{K}_N(t, x, x')| dx' dx.
\end{equation*}
Now, $D$ has finite volume, and we can use formula \eqref{eq:remainder-1}. Up to a $\mathcal{O}(1)e^{\mathcal{O}(|t|)}$ constant, $\| V_N \|'$ is less than the $\sup$ over $x\in D$ and $\alpha\leq 2d+3$ of
\begin{equation*}
\int_{ r  \leq t} \left|\nabla^\alpha_{x,x'}\left\{ (\cosh t - \cosh  r )^{N- \frac{d+2}{2}} \left(-\Delta_{x'} + k^2 - k d\right)u_N \right\} \right| dx'  .
\end{equation*}
Since the curvature of the manifold is bounded by below, we have uniform exponential estimates of the growth of the volume of balls in $\widetilde{M}$, so we can take a $\sup$ of the integrand. Using lemma \ref{lemma:eta-1ound-u_k} and \ref{lemma:r-is-tame}, we deduce that $\Tr V_N(t) = \mathcal{O}(1)e^{\mathcal{O}(|t|)}$.

For the continuity of the remainder, it suffices to observe that the same trace estimates hold for $\partial_t R_N$. 
\end{proof}

\begin{proof}[lemma \ref{lemma:trace-norm}]
We start by recalling the estimate 9.4 from Dimassi-Sj\"ostrand \cite{Dimassi-Sjostrand}. According to this, if $L_1(x,x')$ is the kernel of an operator $L_1:L^2(\R^n) \to L^2(\R^n)$,
\begin{equation*}
\| L_1 \|_{tr} \leq C \sum_{|\alpha|\leq 2n+1} \| \partial_{x,x'}^\alpha L_1 \|_{L^1(\R^n \times \R^n)}.
\end{equation*}
By a partition of unity argument, we obtain a similar estimate for those operators $L_2$ on $M$ whose kernel is supported in $\mathring{M_0}\times \mathring{M_0}$.

Now, by another partition of unity argument, for operators on $M$, we still have to consider three cases. Operators from the cusp to itself, from $M_0$ to the cusp, and vice versa. The three cases can be dealt with using similar techniques, so we will only consider the latter case. With one more partition of unity argument, we can replace $M_0$ by a relatively compact open set $U \subset \subset \R^{d+1}$. Thus, we are trying to estimate the trace norm of an operator $L_3: L^2(Z) \to L^2(U)$ with the kernel supported in $\{(y,\theta)\ |\ y> 2y_0\}\times U$.
 
We take $\chi^0\in C^\infty_c(\R^+, [0,1])$, so that $\chi^0(y)=1$ in $[0,4/3]$ and vanishes outside of $[0,5/3]$. Then let $\chi(y):= \chi^0(y)-\chi^0(2y)$ so that $\chi(y)\geq 0$, $\chi\in C^\infty_c(]2/3, 5/3[)$, equals $1$ on $[5/6,4/3]$ and
\begin{equation*}
\sum_{n\geq 0} \chi(2^{-n}y) = 1,\ \text{for } y>2.
\end{equation*}
We define thus $\chi_n(y,\theta) = \chi( 2^{-n} y/y_0)$. Also let
\begin{equation*}
T_n : f \mapsto \{ (y,\theta) \mapsto 2^{-n d/2}f(2^{n} y,\theta)\}
\end{equation*}
This defines a unitary operator
\begin{equation*}
L^2(Z,\{ 2^n y_0 \leq y \leq 3.2^n y_0\})\to L^2(Z,\{ y_0 \leq y \leq  3y_0\}). 
\end{equation*}
Then
\begin{equation*}
\| L_3 \|_{tr} \leq \sum_{n\geq 0} \| L_3 \chi_n \|_{tr} \leq \sum_{n\geq 0} \| L_3 \chi_n T_n^{-1} \|_{tr}.
\end{equation*}
Each $L_3 \chi_n T_n^{-1}$ is an operator from $\{(y,\theta)\ |\  y_0 \leq y \leq  3y_0\}$ to $U$, and we have already dealt with the case of compact manifolds. It remains to understand
\begin{equation*}
\| \nabla_x^n ( T_n f) \|_{L^1(Z)}.
\end{equation*}
But we have
\begin{equation}\label{eq:transformation-derivatives-translation}
(y\partial_y)^\alpha (y\partial_\theta)^\beta (T_n f) = 2^{-n|\beta|}T_n( (y\partial_y)^\alpha (y\partial_\theta)^\beta f).
\end{equation}
Using the compact case,
\begin{align*}
\| L_3 \chi_n T_n^{-1} \|_{tr} &\leq C \sum_{|\alpha|\leq 2d+3} \| \nabla_{x,x'}^{\alpha} \left\{(T_n)_{x'} L_3 \chi_n\right\}(x,x') \|_{L^1(U\times Z)}.\\
\intertext{Then combining the equivalence of norms, equation \eqref{eq:transformation-derivatives-translation}, and $\|T_n\|_{L^1\to L^1} = 2^{n d/2}$,}
\| L_3 \chi_n T_n^{-1} \|_{tr}	&\leq 2^{nd/2} C \sum_{|\alpha|\leq 2d+3} \| \nabla_{x,x'}^{\alpha} (L_3(x,x')\chi_n(x')) \|_{L^1(U\times Z)}
\end{align*}
Applying the relation \eqref{eq:transformation-derivatives-translation} to $\chi_n$ --- here it is important that $\partial_\theta \chi_n = 0$ --- we conclude that
\begin{equation*}
2^{nd/2} \sum_{n\geq 0} \| \nabla_{x,x'}^{\alpha} (L_3(x,x')\chi_n(x')) \|_{L^1(U\times Z)} \leq C \|{y'}^{d/2} \nabla_{x,x'}^{\alpha} L_3(x,x') \|_{L^1(U\times Z)}.
\end{equation*}
\end{proof}

With lemma \ref{lemma:remainder-Hadamard} in mind, we rewrite the RHS in \eqref{eq:Trace-formula-2} as
\begin{equation}\label{eq:expression-trace-1}
\left\{\lim_{\tau \to \infty} \int_{y\leq e^\tau} \int \psi(t) K_N(t,x,x) dt dx - \kappa \tau \psi(0)\right\} +  \int \psi(t) \Tr R_N(t) dt.
\end{equation}

Let $D_\tau \subset D$ be the part of $D \subset \widetilde{M}$ that projects to $\{ y \leq e^\tau\}$ in $M$. The expression in brackets in the equation above is the limit as $\tau\to \infty$ of
\begin{equation}\label{eq:def-B_tau}
B_\tau :=  \int_{D_\tau} \sum_{\gamma \in \Gamma} \int_\R \psi(t) \widetilde{K}_N(t,x,\gamma x) dt dx  - \kappa \tau \psi(0).
\end{equation}

\subsection{Estimation of contributions to the trace}
\label{sec:contributions-trace}

The main term in $B_\tau$ corresponds to $\gamma = 1$, the diagonal term. To estimate the other terms, we recall some facts on the action of $\Gamma$. 

First, since the injectivity radius of the manifold only goes to $0$ high in the cusp,
\begin{equation*}
R_0 := \inf \left\{d(x,\gamma x) \ \middle| \ x\in \widetilde{M},\ \gamma \neq 1,\ \pi(x,\gamma x) \not\subset Z_\ell,\ \ell=1\dots \kappa\right\} > 0.
\end{equation*}
By $\pi(x,\gamma x)$ we refer to the geodesics on $M$ that lifts to the geodesic between $x$ and $\gamma x$ in $\widetilde{M}$. By $\{\pi(x,\gamma x) \not\subset Z_\ell,\ \ell=1\dots \kappa\}$ we mean that this geodesic of $M$ does not remain in any one cusp.

The universal cover of a cusp is isometric to the open set $U_{y_0}:=\{x= (y,\theta)\ |\ y > y_0\}$ in the half-space model for the hyperbolic space, with the half-space hyperbolic metric $ds^2 = y^{-2} dx^2$. Consider $x\in\widetilde{M}$ such that $d(x,\gamma x) < R_0$ with some $\gamma \neq 1$. Then $x$ and $\gamma x$ have to belong to some open set $\widetilde{U}_0$ of $\widetilde{M}$ isometric to $U_{y_0}$. Since the action of the fundamental group $\Gamma$ is free, $\gamma$ actually preserves $\widetilde{U}_0$, and since it is an isometry of $\widetilde{U}_0$ it acts by translations in the $\theta$ variable.

The set of $\gamma$'s that have such behaviour when restricted to $\widetilde{U}_0$ is the set of $\gamma$'s that preserve $\widetilde{U}_0$, and it is a subgroup of $\Gamma$ isomorphic to $\Lambda_\ell$. The $\ell$ index refers to the index of the cusp onto which $x$ is projected under $\widetilde{M} \to M$. (Recall that $\Lambda_\ell$ is isomorphic to the fundamental group of the cusp $Z_\ell$).

We let $\theta$ be the horizontal coordinate ($\theta \in \R^d$). For future reference, we know that
\begin{equation*}
d((y,0), (y,\theta)) = 2 \arcsinh \frac{|\theta|}{2 y}.
\end{equation*}

Last, we will need the following lemma:
\begin{lemma}\label{lemma:exponential-number-geodesics}
Let $x\in \widetilde{M}$. The number of $\gamma\in\Gamma$ such that $d(x,\gamma x) \leq t$ and such that $\pi(x,\gamma x)$ does not remain in any one cusp, is bounded by $\mathcal{O}(1) e^{\mathcal{O}(|t|)}$, with constants independent of $x$.
\end{lemma} 

This is an elementary result in potential theory for manifolds of negative curvature. For the terminology, we refer to \cite{PPS-12}.
\begin{proof}
Consider the Poincar\'e sum
\begin{equation*}
P_\Gamma(s,x)= \sum_{\gamma \in \Gamma} e^{-s d(x,\gamma x)}.
\end{equation*}
It converges absolutely when $s> \delta_\Gamma$ for all $x\in \widetilde{M}$. This number $\delta_\Gamma> -\infty$ is called the exponent of the group, and it is equal to topological entropy of the manifold. This implies that the number we seek is $\mathcal{O}(1) e^{ (\delta_\Gamma + \epsilon) |t|}$ for any $\epsilon>0$, with a constant depending continuously on $x$ (and $\epsilon$). When $x$ is projected to $M_0$ in $M$, we are done. We have to consider the case when $x$ is projected in some cusp $Z_\ell$. Then, $x$ is in some $\widetilde{U}_0$. Excluding the $\gamma$'s such that $\pi(x,\gamma x)$ does not remain in a cusp corresponds to excluding the $\gamma$'s that preserve $\widetilde{U}_0$. Let the modified sum be:
\begin{equation*}
P_{\Gamma}^\ast(s, x) := \sum_{\gamma \widetilde{U}_0 \cap \widetilde{U}_0=\emptyset} e^{-s d(x,\gamma x)}.
\end{equation*}
One can use comparison of triangles in $\widetilde{M}$ to show that $P_{\Gamma}^\ast(s,x)$ goes to $0$ when $y(x)\to+\infty$ in $\widetilde{U}_0$, and $s> \delta_\Gamma$ is fixed.
\end{proof}

After these preliminaries, we are set to specify which test function $\psi$ we will be considering. Take a function $\rho\in C^\infty_c(]-1,1[)$, even, which equals $1$ around $0$, also take $A>0$. Then let
\begin{equation}\label{eq:def-psi}
\psi(t) = \frac{\sin \lambda t}{\pi t}\rho(At).
\end{equation}
With these assumptions, the last term in \eqref{eq:expression-trace-1} can be bounded:
\begin{equation*}
\int \psi(t) \Tr R_N(t) dt = e^{\mathcal{O}(1/A)}.
\end{equation*} 
For the case of manifolds without conjugate points, we will be considering the regime $A \asymp(1/\log \lambda)$, but until section \ref{sec:conclusion-phase}, $A$ and $\lambda$ can be considered to be independent parameters with $A\leq 1$ and $\lambda\geq 1$.

\subsubsection{The diagonal term}
\label{sec:singularity-at-zero}

We prove 
\begin{lemma}
Assume that $\psi$ takes the form \eqref{eq:def-psi}. Then,
\begin{equation*}
\lim_{\tau \to \infty} \int_{D_\tau} \int \psi(t) \widetilde{K}_N(t,x,x)dt dx = P(\lambda) + \mathcal{O}(1) e^{\mathcal{O}(1/A)}
\end{equation*}
where $P(\lambda)= c_0 \lambda^{d+1} + \dots + c_k \lambda^{d+1-2k} + \dots$.
\end{lemma}
This estimate is equivalent to an integrated version of formula (56) in \cite{Berard-77}, and the proof is similar.

\begin{proof}
First, we write out the LHS in the lemma as the sum of
\begin{equation*}
\frac{C_0}{(-2)^k} \lim_{\tau \to \infty} \int_{D_\tau} u_k(x,x) \int  \psi(t)\sinh |t|  M_{k-\frac{d+2}{2}}( \cosh t - \cosh  r )\Big|_{r=0} dt dx. 
\end{equation*}
for $k= 0,\dots, N$. This is
\begin{equation*}
\frac{C_0}{(-2)^k} \lim_{\tau \to \infty} \int_{D_\tau} u_k(x,x) dx \times\int \frac{\sin t \lambda}{\pi t} \rho(At) \sinh|t| M_{k-\frac{d+2}{2}}(\cosh t - 1) dt. 
\end{equation*}
The functions $x \mapsto u_k(x,x)$ are bounded, hence integrable. When $k > d/2 $, the whole integral can be bounded above directly by 
\begin{equation*}
\int_0^{1/A} e^{|t|(k- d/2)} dt = \mathcal{O}(1) e^{\mathcal{O}(1/A)}.
\end{equation*}
When $k \leq d/2 $, let us separate the time integral into two parts, using $1 = \rho(t) + 1-\rho(t)$. The part of the integral supported away from $0$ is $\mathcal{O}(1 + |\log A|)$ as $\lambda \to + \infty$, as can be seen taking the $L^1$ norm of the integrand. We are left with the following integrals ($k\leq d/2$)
\begin{equation*}
\int \frac{\sin t \lambda}{\pi t} \rho(t) \sinh|t| M_{k-(d+2)/2}(\cosh t - 1) dt. 
\end{equation*}
When $t\to 0$, $\cosh t - 1 \sim t^2/2$, so this is the Fourier transform of 
\begin{equation*}
\epsilon(t) M_{k-d/2 - 1}(t^2)\times W(t),
\end{equation*}
where $W\in C^{\infty}_c(]-1,1[)$ is even and $\epsilon(t)$ is the sign of $t$ --- one can check as for $|t| M_\alpha(t^2 - r^2)$ that this a well defined distribution. Since $\epsilon(t) M_{k-d/2 - 1}(t^2)$ is a homogeneous distribution of order $2k-d-2$, and since $W$ is even, we have an expansion of the Fourier transform with powers of $\lambda$ of the same parity, starting with $\lambda^{d+ 1 - 2k}$.

\end{proof}

\subsubsection{The contribution of the cusps}
\label{sec:cusp-contribution}

From the discussion at the start of section \ref{sec:contributions-trace}, if $d(x,\gamma x) < R_0$, then the projection of $x$ in $M$ is in some $Z_\ell$. In that case, $\gamma$ can be identified with an element of $\Lambda_\ell$. We can compute $\cosh(d(x,\gamma x)) = |\gamma|^2/(2y) - 1$ where $|\cdot|$ is understood as the norm on $\Lambda_\ell \subset \R^d$. Since the corresponding geodesic remained in an open set of curvature $-1$, in the contribution to \eqref{eq:expression-trace-1}, only the first term of the Hadamard parametrix is present. In particular, in each cusp, we have contributions from $\gamma = 1$, that we have already computed, contributions from $\gamma$'s for which $d(x,\gamma x)\geq R_0$ that we will estimate, and a specific contribution created by the cusp:
\begin{equation}\label{eq:expression-trace-cusp}
\lim_{\tau \to \infty} - \psi(0)\tau + \int_{y_0}^{e^\tau} \frac{C_0 dy}{y^{d+1}} \sum_{\substack{\gamma \in \Lambda_\ell,\\ \gamma \neq 0}} \int \psi(t) \sinh|t| M_{-\frac{d}{2} - 1}\left[\cosh t - \frac{|\gamma|^2}{2y^2} - 1\right] dt.
\end{equation}

\begin{lemma}
If one replaces $y_0$ by $0$ in equation \eqref{eq:expression-trace-cusp} the result is
\begin{equation*}
-\frac{\lambda}{\pi} \log \lambda + \frac{\mathcal{C}_1(\Lambda_\ell)}{\pi} \lambda +\mathcal{O}(1).
\end{equation*}
Here, the constant $\mathcal{C}_1(\Lambda_\ell)$ only depends on the lattice and $\mathcal{C}_1(\Z)= 1-\log 2$. When $d>1$, this contribution will be smaller than the $\mathcal{O}(\lambda^d/\log \lambda)$ remainder.
\end{lemma}

\begin{proof}
We consider the big integral in \eqref{eq:expression-trace-cusp}, with $y_0$ replaced by $0$ and first make the change of variables $u= |\gamma|/y$. We find a new expression
\begin{equation*}
C_0 \sum_{\gamma \in \Lambda_\ell, \ \gamma\neq 0} \frac{1}{|\gamma |^d} \int_{|\gamma|e^{-\tau}}^{+\infty} u^{d-1} du \int \psi(t) \sinh|t| M_{-d/2 - 1}\left( \cosh t - 1  - \frac{u^2}{2} \right) dt. 
\end{equation*}
Let $\widetilde{\psi}(v_0) = \psi(t(v_0))$, where $\cosh t - 1 = v_0$, and then $v= v_0 - u^2/2$. We can rearrange the above expression in the following way (beware of the integration by parts in $v$):
\begin{equation}\label{eq:contrib-cusp-1}
- 2 C_0 \int M_{-d/2}\left(v \right) \int u^{d-1} \widetilde{\psi}'\left(v + \frac{u^2}{2}\right) \sum_{\gamma \neq 0,\  |\gamma| \leq u e^\tau}\frac{1}{|\gamma |^d} du dv.
\end{equation}
Using simple arguments of comparison between series and integrals, one finds that for fixed $u \neq 0$ and $\tau \to + \infty$,
\begin{equation}\label{eq:def-gamma}
\sum_{\gamma \neq 0,\ |\gamma| \leq u e^\tau} \frac{1}{|\gamma|^d} = \frac{2 \pi^{d/2}}{\Gamma(d/2)} \Big\{\log(u e^\tau) + \overline{\gamma}(\Lambda_\ell) + o(1)\Big\},
\end{equation}
where $\overline{\gamma}(\Lambda_\ell)$ is a constant depending only on $\Lambda_\ell$ (when $d=1$, it is just the Euler-Mascheroni constant). As $\tau \to \infty$, \eqref{eq:contrib-cusp-1} takes the form $a\tau + b + o(1)$ by an argument of dominated convergence. The expression of $a$:
\begin{equation*}
-2 C_0 \frac{2 \pi^{d/2}}{\Gamma(d/2)} \int M_{-d/2}\left(v \right) \int u^{d-1} \widetilde{\psi}'\left(v + \frac{u^2}{2}\right) du dv.
\end{equation*}
Now, let $w= u^2/2$. The last expression becomes
\begin{equation*}
- \int M_{-d/2}\left(v \right) \int M_{d/2 -1}(w)\widetilde{\psi}'\left(v + w\right)  dw dv.
\end{equation*}
Observe that integrating by parts $2m$ times, this is equal to
\begin{equation*}
- \int M_{-d/2 + m}(v) \int M_{d/2 -m -1} (w) \widetilde{\psi}'(v+w) dw dv.
\end{equation*}
Now, we make a distinction. If $d$ is even, taking $m=d/2$, since $M_{-1} = \delta$, this reduces to $\widetilde{\psi}(0) = \psi(0)=\lambda/\pi$. When $d$ is odd, the result is the same. Indeed, taking $m=(d-1)/2$, we find
\begin{equation*}
-\frac{1}{\pi}\int_{v>0, w>0} \frac{1}{\sqrt{vw}} \widetilde{\psi}'(v+w) dvdw = -\frac{1}{\pi}\int_0^{+\infty} \widetilde{\psi}'(z)\int_0^z \frac{dv}{\sqrt{v(z-v)}} = \widetilde{\psi}(0).
\end{equation*}

We get the confirmation that the divergence as $\tau \to + \infty$ is created at the cusps. Now, remember we are looking for the result of the re-normalization, that is, the constant $b$. Its expression is
\begin{equation*}
-2 C_0 \frac{2 \pi^{d/2}}{\Gamma(d/2)} \int M_{-d/2}\left(v \right) \int u^{d-1} \widetilde{\psi}'\left(v + \frac{u^2}{2}\right) \left[\log u + \overline{\gamma}(\Lambda_\ell)\right]du dv.
\end{equation*}
Changing again the parameter with $w=u^2/2$, this is found equal to 
\begin{equation*}
-\int M_{-d/2}\left(v \right) \int M_{d/2 -1}(w)\widetilde{\psi}'\left(v + w\right) \left[\frac{1}{2}\log w + \frac{1}{2}\log 2 + \overline{\gamma}(\Lambda_\ell) \right] dw dv.
\end{equation*}
The constant term contributes in the final expression of $b$ by
\begin{equation}\label{eq:contrib-b-1}
a \times \left[ \frac{1}{2} \log 2 + \overline{\gamma}(\Lambda_\ell) \right] = \frac{\lambda}{\pi}\left[ \frac{1}{2} \log 2 + \overline{\gamma}(\Lambda_\ell) \right].
\end{equation}
On the other hand, in the contribution from the $\log w$ term, change variables again with $v+w = V$, and $w=Vx$ with $x\in [0,1]$, and $V \in \R^+$. The contribution to $b$ becomes
\begin{equation*}
-\frac{1}{2}\int \widetilde{\psi}'\left( V \right) \int M_{-d/2}\left(1-x \right)  M_{d/2 - 1}(x)\log(Vx) dx dV
\end{equation*}
This integral is well defined as we are taking the product of distributions that are not singular at the same points, and the result is a compactly supported distribution (in $x$). From here, the next step is to compute the integrals
\begin{equation}\label{eq:two-integrals}
\int M_{-d/2}\left(1-x \right)  M_{d/2 - 1}(x)\log(x) dx\text{ and }\int M_{-d/2}\left(1-x \right)  M_{d/2 - 1}(x) dx.
\end{equation}
After integrating by parts, the first $x$ integral becomes, when $d$ is even,
\begin{equation*}
2\mathcal{C}(d):=\frac{1}{d/2 - 1} + \dots + \frac{1}{2} + 1 = 2 \sum_{k=1}^{d/2-1} \frac{1}{d-2k}.
\end{equation*}
On the other hand, when $d$ is odd, it is
\begin{equation*}
2\mathcal{C}(d):=2\sum_{k=1}^{\lfloor d/2 \rfloor} \frac{1}{d-2k} + \frac{1}{\pi} \int \frac{\log x dx}{\sqrt{x(1-x)}} = 2\sum_{k=1}^{\lfloor d/2 \rfloor} \frac{1}{d-2k} - 2 \log 2.
\end{equation*}
($\mathcal{C}(1) = - \log 2$.) The final contribution to $b$ will be 
\begin{equation}\label{eq:contrib-b-2}
\mathcal{C}(d)\widetilde{\psi}(0) = \frac{\mathcal{C}(d) \lambda}{\pi}.
\end{equation}

The second integral in \eqref{eq:two-integrals} one can be computed as was $a$ and is equal to $1$. We claim that the corresponding contribution to $b$ is
\begin{equation}\label{eq:main-contrib-cusp}
J_b :=- \frac{1}{2} \int_{0}^{+\infty}\hspace{-10pt} \widetilde{\psi}'(V)\log V dV = -\frac{\lambda \log \lambda}{\pi} - \frac{\lambda \log 2}{2\pi} + \frac{\lambda}{\pi}(1-\overline{\gamma}(\Z)) + \mathcal{O}(1).
\end{equation}
We come back to the $t$ variable in the integral, with $\cosh t - 1 = V$. 
\begin{equation*}
J_b =-\frac{1}{2\pi} \int_{0}^{+\infty} \frac{d}{dt}\left\{ \frac{\sin \lambda t}{t} \rho(A t) \right\} \log \left[2 \sinh^2\frac{t}{2}\right] dt. 
\end{equation*}
This gives a $\lambda\log 2 /(2\pi)$ term, and a term
\begin{equation*}
-\frac{1}{\pi} \int_{0}^{+\infty} \frac{d}{dt}\left\{ \frac{\sin \lambda t}{t} \rho(A t) \right\} \log \left[\sinh\frac{t}{2}\right] dt.
\end{equation*}
We can insert in the differentiated expression $1=\rho(t) + (1-\rho(t))$. The second term is a $o(1/\lambda)$, as we recognize the Fourier transform of a smooth, $L^2$ function, whose derivative is in $L^1$. We are left with
\begin{equation*}
-\frac{1}{\pi} \int_{0}^{+\infty} \frac{d}{dt}\left\{ \frac{\sin \lambda t}{t} \rho(t) \right\} \log \left[\sinh\frac{t}{2}\right] dt.
\end{equation*}
Now, we change variables $u= t\lambda$, integrate by part and find
\begin{equation*}
\frac{1}{\pi} \lim_{\epsilon \to 0^+}\left\{ \int_\epsilon^{+\infty} \frac{1}{2}\frac{\sin u}{u} \rho(u/\lambda) \coth \frac{u}{2\lambda} du + \lambda\frac{\sin \epsilon}{\epsilon} \log \left[\sinh\frac{\epsilon}{2\lambda}\right]\right\}.
\end{equation*}
We recover the main term $- \lambda\log(2\lambda)/\pi$, and
\begin{equation*}
\frac{1}{\pi} \lim_{\epsilon \to 0^+}\left\{ \int_\epsilon^{+\infty} \frac{1}{2}\frac{\sin u}{u} \rho(u/\lambda) \coth \frac{u}{2\lambda} du + \lambda\log (\epsilon)\right\}.
\end{equation*}
But, as $\lambda \to \infty$,
\begin{equation*}
\frac{1}{2}\rho(u/\lambda)\coth \frac{u}{2\lambda} = \frac{\lambda}{u} + (1-\rho(u/\lambda))\frac{\lambda}{u} + \mathcal{O}(1) \rho(u/\lambda)\frac{u}{\lambda}.
\end{equation*}
Both the $(1-\rho)$ and the $u/\lambda$ term will only contribute to $b$ by $\mathcal{O}(1)$. So the last integral we have to compute is
\begin{equation*}
\frac{\lambda}{\pi} \lim_{\epsilon \to 0^+}\left\{ \int_\epsilon^{+\infty} \frac{\sin u}{u^2} du + \log (\epsilon)\right\}.
\end{equation*}
To compute this last constant, one may use Cauchy's theorem, shifting the contour of integration to $i\R^+$; this gives $1-\overline{\gamma}(\Z)$. Then, summing contributions from \eqref{eq:contrib-b-1}, \eqref{eq:contrib-b-2} and \eqref{eq:main-contrib-cusp}, the value of $b$ is found to be 
\begin{equation*}
-\frac{\lambda}{\pi}\log \lambda +\frac{\lambda}{\pi}\left[\cancel{\frac{\log 2}{2}} + \overline{\gamma}(\Lambda_\ell) + \mathcal{C}(d) - \cancel{\frac{\log 2}{2}} + 1 - \overline{\gamma}(\Z) \right] + \mathcal{O}(1).
\end{equation*}

In particular, for $d=1$, we do get the coefficient $1 - \log 2$ for $\lambda$ (which is the value computed by Selberg).
\end{proof}

\subsubsection{The other terms}
\label{sec:other-contributions}

There are two contributions to $B_\tau$ left to compute. The first one is
\begin{equation}\label{eq:remainder-terms-1}
C_0 \int_0^{y_0} \frac{dy}{y^{d+1}} \sum_{\gamma \in \Lambda_\ell,\ \gamma \neq 0} \int \psi(t) \sinh|t| M_{-d/2 - 1}\left[\cosh t - 2\left(\frac{|\gamma|}{2y}\right)^2 - 1\right] dt.
\end{equation}
The other one is the sum over $k=1, \dots, N$ of (up to some constants)
\begin{equation}\label{eq:remainder-terms-2}
\int_{D} \sum_{\substack{\gamma \neq 1,\\ d(x,\gamma x) > R_0}} u_k(x,\gamma x) \int \psi(t) \sinh|t| M_{k-\frac{d}{2} - 1}(\cosh t - \cosh d(x,\gamma x)) dt dx.
\end{equation}

These are \emph{remainder terms}, as we will see. The arguments we use are adapted from B\'erard. The first step is the following: assume $R_0< R < 1/A$, then we consider
\begin{equation*}
I:=\int_{t>0} \psi(t) \sinh|t| M_{k-d/2 - 1}(\cosh t - \cosh R) dt.
\end{equation*}
We insert the cut-off $1 = \rho(t-R) + 1-\rho(t-R)$. The part $1-\rho(t-R)$ only contributes $\mathcal{O}(\lambda^{-\infty})e^{\mathcal{O}(1/A)}$ (it suffices to integrate by parts, and recall that the only $R$'s that contribute are $\mathcal{O}(1/A)$). We are left with the $\rho(t-R)$ part. On the interval where this integral is supported, we can write 
\begin{equation*}
\cosh t - \cosh R = (t-R) \sinh R W(t-R,R),
\end{equation*}
with $W(t-R,R)$ smooth uniformly bounded, not vanishing (and $W(0,R) = 1$). Hence we have another function $\widetilde{W}(u,R)$ satisfying similar assumptions such that
\begin{equation*}
M_{k-d/2-1}(\cosh t - \cosh R) = \sinh(R)^{k-d/2-1} \widetilde{W}(t-R,R) M_{k-d/2 -1}(t-R).
\end{equation*}
As a consequence, $I$ is the imaginary part of
\begin{equation*}
\frac{\sinh(R)^{k-\frac{d}{2}}}{R} \int e^{i\lambda(R+u)} \frac{\rho(A(R+u))\rho(u)\sinh(R + u)}{\sinh(R)(1+\frac{u}{R})}  \widetilde{W}(u, R) M_{k-\frac{d}{2}-1}(u) du.
\end{equation*}
We deduce that when $R_0 < R < 1/A$,
\begin{equation*}
I = \mathcal{O}(\lambda^{d/2 -k}) e^{\mathcal{O}(1/A)}.
\end{equation*}
This is the equivalent of estimate (59) in B\'erard (actually, $I$ has an expansion in powers of $\lambda^{-1}$).

Now, we can estimate the contributions \eqref{eq:remainder-terms-1} and \eqref{eq:remainder-terms-2} to $B_\tau$. First for \eqref{eq:remainder-terms-1}, that contribution is bounded by
\begin{equation*}
\begin{split}
\int_0^{y_0} \frac{dy}{y^{d+1}}& \sum_{\gamma \in \Z^d} \mathcal{O}(\lambda^{\frac{d}{2}})  e^{\mathcal{O}(1/A)} \mathbb{1}\left\{|\gamma|\leq 2 y \cosh (1/2A)\right\}\\
	& = \mathcal{O}(\lambda^{\frac{d}{2}}) e^{\mathcal{O}(1/A)} \sum_{|\gamma| < 2 y_0 \cosh(1/(2A))} \int_{|\gamma|/(2\cosh(1/2A))}^{y_0} \frac{dy}{y^{d+1}}.
\end{split}
\end{equation*}
Let $L= 1/(2\cosh(1/2A))$. It suffices to see that
\begin{equation*}
\sum_{L |\gamma| < y_0} \int_{L|\gamma|}^{y_0} \frac{dy}{y^{d+1}} = \mathcal{O}(L^{-d}) \sum_{0<|\gamma| \leq y_0/L} |\gamma|^{-d} = \mathcal{O}(L^{-d}\log L) = \mathcal{O}(1)e^{\mathcal{O}(1/A)}.
\end{equation*}

Now, to estimate \eqref{eq:remainder-terms-2}, we use the fact that the number of non vanishing terms in the sum is $\mathcal{O}(1)e^{\mathcal{O}(1/A)}$ according to lemma \ref{lemma:exponential-number-geodesics}, and lemma \ref{lemma:eta-1ound-u_k} directly to find that it contributes by $\mathcal{O}(\lambda^{d/2}) e^{\mathcal{O}(1/A)}$ to $B_\tau$.

\subsection{The conclusion}
\label{sec:conclusion-phase}

We can now complete the proof of the third part of theorem \ref{thm:Phase}.

\subsubsection{When there are no conjugate points}

When there are no conjugate points, we gather all the different pieces above:
\begin{align}
\label{eq:regularized-counting}	
\sum_{\mu_n\in\sigma_{pp}(-\Delta)} \widehat{\psi}&\left(\sqrt{\mu_n - \frac{d^2}{4}}\right) -  \frac{1}{2} \int \Psi'  \hat{\psi}+ \frac{1}{4} \hat{\psi}(0) \Tr \phi\left(\frac{d}{2}\right) \\ 
												&= \lim_{\tau \to \infty} \int_{y \leq e^\tau} \int_\R \psi(t) K(t,x,x)dx dt -\kappa \tau \psi(0), \notag \\
												&= \lim_{\tau \to \infty} \int_{y \leq e^\tau} \int_\R \psi(t) K_N(t,x,x)dx dt -\kappa \tau \psi(0) + \mathcal{O}(1)e^{\mathcal{O}(1/A)}, \notag \\
												&= \sum_{k \geq 0}^{d/2} c_k \lambda^{d+1-2k} -  \frac{\kappa \lambda}{\pi} \log \lambda + \frac{\kappa(1-\log 2)\lambda}{\pi} + \mathcal{O}(\lambda^{\frac{d}{2}})e^{\mathcal{O}(1/A)}. \notag
\end{align}

When taking $A$ to be a sufficiently large multiple of $1/\log \lambda$, we can get $e^{\mathcal{O}(1/A)} = \mathcal{O}(\lambda^\epsilon)$ for any fixed $\epsilon >0$. Recall $\widetilde{N}(\lambda)$ is the counting function defined in \eqref{eq:def-tilde-N}:
\begin{equation*}
\widetilde{N}(\lambda) = \sum_{\mu_n \in \sigma_{pp}(-\Delta)} \mathbb{1}\left(\sqrt{\mu_n - \frac{d^2}{4}} \leq \lambda\right) - \frac{1}{2} \int_{-\lambda}^\lambda \Psi'.
\end{equation*}
From the definition of $\psi$, we deduce that the quantity in the LHS of \eqref{eq:regularized-counting} is, up to $\mathcal{O}(1)$,
\begin{equation*}
\int \widetilde{N}(\lambda+ u) \frac{1}{A}\widehat{\rho}\left( \frac{u}{A}\right)du.
\end{equation*}
To recover $\widetilde{N}$, the strategy is to control variations of $\widetilde{N}$ on scales of size $A$ using \eqref{eq:regularized-counting}. First, consider $\chi\in C^\infty_c(]-1/2,1/2[)$ with $\|\chi\|_{L^2} = 1$. Then we can take $\rho= \chi \ast \chi$. In that case, $\widehat{\rho}= \widehat{\chi}^2\geq 0$.  From equation \eqref{eq:decomp-increasing-phase} (in the appendix), we deduce that $\widetilde{N} = P + f$ where $P$ is a $\mathcal{O}(\lambda^{d+1})$ polynomial and $f$ is an increasing function. As a consequence, for $x> 0$,
\begin{align*}
\int [f(\lambda+ x + A/2 + u) &- f(\lambda + u)] \frac{1}{A} \widehat{\rho}\left(\frac{u}{A}\right) du  \\
			&\geq \int_{-1/2}^0 [f(\lambda+x +A(1/2+u))- f(\lambda+Au)] \widehat{\rho}(u)du \\
			& \geq c(f(\lambda+x) - f(\lambda)).
\end{align*}
whence we deduce that when $\lambda_2 > \lambda_1$,
\begin{align*}
&|f(\lambda_2) - f(\lambda_1)| \\
	&\leq \frac{1}{c}\left[ \sum_{k \geq 0}^{d/2} c_k \lambda'^{d+1-2k} -  \frac{\kappa \lambda'}{\pi} \log \lambda' + \frac{\kappa(1-\log 2)\lambda'}{\pi} + \mathcal{O}(\lambda'^{\frac{d}{2}})e^{\mathcal{O}(\frac{1}{A})}\right]_{\lambda_1}^{\lambda_2 + \frac{A}{2}}\\
	&= \mathcal{O}( \lambda_1^d/\log \lambda_1 + |\lambda_2 - \lambda_1| \lambda_1^d + |\lambda_2 - \lambda_1|^{d+1}).
\end{align*}
Using the same argument for $x<0$, we can swap the role of $\lambda_2$ and $\lambda_1$ in the last inequality, i.e remove the condition $\lambda_2 > \lambda_1$. The same estimate holds for $\widetilde{N}(\lambda)$, and we deduce
\begin{equation*}
\widetilde{N}(\lambda) - \int \widetilde{N}(\lambda+Au)\widehat{\rho}(u)du = \mathcal{O}(1)\int \widehat{\rho}(u)[ A\lambda^d |u| + A^{d+1} |u|^{d+1} + A \lambda^d] du, 
\end{equation*}
and that is $\mathcal{O}(\lambda^d/\log \lambda)$. This ends the proof of the third estimate in theorem \ref{thm:Phase}. Indeed, to identify the constant $c_0$ in the RHS of \eqref{eq:regularized-counting}, it suffices to check the equivalent given by M\"uller in proposition 4.13 of \cite{Muller-86}. That is,
\begin{equation*}
c_0 = \frac{\vol M}{(4\pi)^{(d+1)/2} \Gamma(d/2 + 3/2)}= \frac{\vol B^\ast M}{(2\pi)^{d+1}}.
\end{equation*}

\subsubsection{Without assumptions}
\label{sec:no-assumption-curvature}

Now, we do not assume any more that there are no conjugate points. First, we explain why theorem \ref{thm:singularities-trace} holds. Then we give the usual H\"ormander-Levitan bound on the remainder in all generality, and then, we turn to the aperiodic case.

\textbf{The singularities of the wave trace.} For compact manifolds, theorem \ref{thm:singularities-trace} is due to Chazarain \cite{Chazarain-74}, see also \cite{Duistermaat-Guillemin-75}. The proof relies on the fact that $e^{it\sqrt{-\Delta + c}}$ is an FIO, micro-supported on the graph of the geodesic flow. The proof being essentially local, it can be worked out on a manifold with cusps. There is no obstruction to localizing the argument since closed geodesics with length less than some constant live in a compact part of the manifold. The result of these considerations is two-fold. 

First, we recover that periodic geodesic may contribute singularities to the wave trace, at the times corresponding to their algebraic lengths. We also find a singularity at $0$.

We also recover that any other singularity must come from the behaviour at infinity in the cusp. But we have seen that cusps merely modify the usual form of the singularity at zero, and do not create new singularities, and this observation ends the proof.

\textbf{A H\"ormander type remainder.} The compact part of the manifold always has a positive injectivity radius $r>0$. Hence, we can still build a Hadamard parametrix for times $|t| < r$. One can check that all the arguments above apply, albeit replacing $A = c/\log \lambda$ by a fixed $A >0$ sufficiently large. From there, one deduces without assumptions on the curvature, 
\begin{equation*}
\int \widetilde{N}(\lambda+u) \widehat{\rho}(u) du = c \lambda^{d+1} - \frac{\kappa \lambda}{\pi} \log \lambda + \mathcal{O}(\lambda^d).
\end{equation*}
Using the same argument as above, one finds $\widetilde{N}(\lambda)= c \lambda^{d+1} + \frac{\lambda}{\pi} \log \lambda + \mathcal{O}(\lambda^d)$, which proves part 1 of theorem \ref{thm:Phase}.

\textbf{The aperiodic case.} Let us assume now that the flow is aperiodic, i.e, the set of closed geodesics has measure zero. Let us take $T>0$. Let $F_T$ be the set of $\xi \in S^\ast M$ periodic under the geodesic flow with period $|t|\leq T$. Since periodic geodesic of length at most $T$ cannot intersect the part of cusps $\{y> e^{T/2}y_0\}$, $F_T$ is a closed and compact set.

Given $\epsilon >0$, we can find a function $b\in C^\infty_c(S^\ast M)$ such that $b=1$ on a neighbourhood of $F_T$, and $\int |b|^2 \leq \epsilon$. We can assume that $b$ takes values in $[0,1]$. We see $b$ as a $0$-homogeneous function on $T^\ast M$, and let $\hat{b} = \Op(b)$ --- using a quantization such that $\hat{b}$ is compactly supported. We can then find another pseudo-differential operator $\hat{B}$ such that $\hat{B}^\ast \hat{B} + \hat{b}^\ast \hat{b} = \mathbb{1} + \mathcal{O}(\Psi^{-\infty})$. Observe that we can impose that the remainder is compactly supported. The principal symbol of $\hat{B}$ is $B$ such that $|b|^2 + |B|^2 = 1$.

Now, we can follow Ivrii's argument \cite{Ivrii-80}. We are trying to determine 
\begin{equation*}
\OTr \mathbb{1}(\sqrt{-\Delta - d^2/4} \leq \lambda),
\end{equation*}
so we insert the relation $\mathbb{1}  = \hat{B}^\ast \hat{B} + \hat{b}^\ast \hat{b} + \mathcal{O}(\Psi^{-\infty})$ inside the trace, to cut it into three parts $\widetilde{N}_B(\lambda)$ and $\widetilde{N}_b(\lambda)$, plus a remainder that is $\mathcal{O}(1)$:
\begin{align*}
\widetilde{N}_b(\lambda)&= \Tr \mathbb{1}(\hat{b}^\ast \hat{b} \sqrt{-\Delta - d^2/4} \leq \lambda),\\ \widetilde{N}_B(\lambda)&= \OTr \mathbb{1}(\hat{B}^\ast \hat{B} \sqrt{-\Delta - d^2/4} \leq \lambda).
\end{align*}
Usual arguments show that $\widetilde{N}_b(\lambda)$ is well defined, as is the trace of the smoothing compactly supported remainder. For $\widetilde{N}_B$, one may observe that the $0$-Trace is well defined because in a small neighbourhood of the cusp, $B$ acts exactly as the identity. In particular the arguments from section \ref{sec:preliminaries} carry out here almost directly.

The FIO techniques used in \cite{Duistermaat-Guillemin-75} are compatible with the use of pseudo-differential operator by design. In particular the Fourier transforms of $\widetilde{N}_b'$ and $\widetilde{N}_B'$ are only singular at $0$, and at times corresponding to algebraic lengths of geodesics intersecting their micro-support. The type of singularity at $0$ is constrained. For $t_0>0$, let $\rho_{t_0}(t)=\rho(t/t_0)$. Combining these arguments from \cite{Duistermaat-Guillemin-75} and the results on the singularity at $0$ given by the cusp computed in \ref{sec:cusp-contribution}, we find
\begin{equation*}
\widetilde{N}_B\ast \widehat{\rho_T} (\lambda)= a_0^B \lambda^{d+1} + a_1^B \lambda^{d} + \dots - \frac{\kappa}{\pi} \lambda \log \lambda + \frac{\kappa(1-\log 2)}{\pi}\lambda + \mathcal{O}(1)
\end{equation*}
and with $T_0$ smaller than the smallest length of periodic geodesic,
\begin{equation*}
\widetilde{N}_b\ast \widehat{\rho_{T_0}} (\lambda)= a_0^b \lambda^{d+1} + a_1^b \lambda^{d} + \dots + \mathcal{O}(\lambda^{-\infty}).
\end{equation*}
Additionally --- recall that the powers in the singularity at $0$ have all the same parity --- 
\begin{equation*}
a_0^B = \int_{B^\ast M} B^2,\ a_0^b= \int_{B^\ast M} b^2,\ a_1^B + a_1^b = 0.
\end{equation*}

The function $\widetilde{N}_b$ is non-decreasing. In particular, we can refine the argument used for the case of no-conjugate points. We have a constant $C$ depending only on $\rho$ --- which may change at every line --- so that for $u\geq 0$, 
\begin{equation*}
|\widetilde{N}_b(\lambda+u) - \widetilde{N}_b(\lambda)| \leq C [\widetilde{N}_b\ast \widehat{\rho_{T_0}}]_{\lambda}^{\lambda + u + A/2},
\end{equation*}
so that
\begin{equation*}
| \widetilde{N}_b(\lambda)- \widetilde{N}_b\ast \widehat{\rho_{T_0}} (\lambda)| \leq C \int \widehat{\rho}(u) | \widetilde{N}_b\ast \widehat{\rho_{T_0}}(\lambda + (u+1/2)/T_0) - \widetilde{N}_b\ast \widehat{\rho_{T_0}}(\lambda) |du.
\end{equation*}
We deduce that
\begin{equation*}
|\widetilde{N}_b(\lambda) - a_0^b \lambda^{d+1} - a_1^b \lambda^d | \leq \frac{C a_0^b}{T_0} \lambda^d + \mathcal{O}(\lambda^{d-1}).
\end{equation*}

Using cut-offs, we can decompose $\hat{B}^\ast \hat{B}$ as the sum of two operators, one supported for $y > e^{T}y_0$ and the other for $y< 2e^{T}y_0$. The contribution to $\widetilde{N}_B$ from the compactly supported one is non-decreasing. For the other one, we can use formulas \eqref{eq:Maass-Selberg-bis}  and modify formula \eqref{eq:trace-cutoff} to express it as a sum of two terms. One involves non-constant Fourier modes in $\theta$ and is a non-decreasing function of $\lambda$. The other involves the constant Fourier mode of the Eisenstein series and is a quantity of the form $-\kappa T \lambda /\pi + I(T, \lambda)$. Here $I$ is an integral reminiscent of $I(\tau)$ introduced in the proof of formula \ref{eq:trace-cutoff}, involving the trace of the scattering matrix $\phi(s)$. In the end, $I$ is $\mathcal{O}(1)$ when $\lambda \to +\infty$. As a consequence, the arguments above concerning $\widetilde{N}_b$ also apply to $\widetilde{N}_B$ on the time scale $T$:
\begin{equation*}
\begin{split}
&\left|\widetilde{N}_B(\lambda) - a_0^B \lambda^{d+1} - a_1^B \lambda^d + \frac{\kappa}{\pi} \lambda \log \lambda - \frac{\kappa(1-\log 2)}{\pi}\lambda \right| \\
 &\qquad\qquad\qquad\qquad \leq \frac{C a_0^B}{T}\lambda^d + \mathcal{O}(\lambda^{d-1} + \log \lambda).
\end{split}
\end{equation*}
However, $a_0^b = C\int |b|^2 \leq C\epsilon$. Hence we find
\begin{equation*}
\begin{split}
&\left|\widetilde{N}(\lambda) - a_0 \lambda^{d+1} + \frac{\kappa}{\pi} \lambda \log \lambda - \frac{\kappa(1-\log 2)}{\pi}\lambda\right| \\
& \qquad\qquad\qquad\qquad\leq  C\left( \frac{1}{T} + \epsilon\right) \lambda^d + \mathcal{O}(\lambda^{d-1} + \log \lambda).
\end{split}
\end{equation*}
Since $T$ and $\epsilon$ were arbitrary, this proves estimate \eqref{eq:continuous-estimate-periodic-zero} in the theorem.

\section{Counting resonances in negative curvature}
\label{sec:Counting}

Now, we turn to the proof of theorem \ref{thm:Resonances-negative}. Throughout this section, we assume that the curvature of $g$ is negative in $M$. According to the relation \eqref{eq:symmetry-varphi}, we deduce that counting the poles of $\varphi(s)$ in $\{\Re s < d/2\}$ is the same as counting the zeroes of $\varphi$ in $\{\Re s >d/2\}$. So, from now on, we will work in $\{\Re s > d/2\}$, and will be interested in the zeroes of $\varphi$.

A Dirichlet series is a formal series of the form
\begin{equation*}
L(s)= \sum_{k=0}^\infty \frac{ a_k}{\lambda_k^s},
\end{equation*}
where the $\lambda_k$'s form an increasing sequence of positive real numbers, and the $a_k$'s are real numbers. We will only consider convergent Dirichlet series, i.e such series that converge in some half plane $\Re s > s_0$. From \cite{Bonthonneau-3}, we recall that there are constants $\delta(g)>d/2$, $T^\#\in \R$ and Dirichlet series $L_0$, $L_1$, $\dots$, that converge absolutely for $\Re s > \delta(g)$ and such that when $\Re s > \delta(g)$
\begin{equation}\label{eq:parametrix-phi}
\varphi(s) = s^{-\kappa d/2}\left(L_0(s) + \frac{1}{s} L_1(s) + \dots + \mathcal{O}\left(\frac{ e^{ - s T^\#}}{s^\infty} \right) \right).
\end{equation}
We can denote the coefficients of the $L_j$'s in the following way:
\begin{equation}\label{eq:def-a^i_k}
L_j(s) = \sum_k a^j_k e^{- s \ell_k},\text{ with }\ell_0 < \ell_1 <\dots, \text{ for } \Re s > \delta_g.
\end{equation}
It is possible that $a^0_0$ vanishes, so we let $n^\ast(g)$ be the smallest integer $n$ such that $a^0_n \neq 0$. Now, we set
\begin{equation}\label{eq:def-frak{G}}
\mathfrak{G}(M):= \{ g\ |\ n^\ast(g) < \infty\} = \{g\ |\ L_0(s) \text{ is not identically }0\}.
\end{equation}
For metrics in $\mathfrak{G}(M)$, we can define $\ell_\ast = \ell_{n^\ast}$ and $a^0_\ast = a^0_{n^\ast}$. 
According to lemma 4.5 in \cite{Bonthonneau-3}, $\mathfrak{G}(M)$ is open is $C^2$ topology on metrics. According to lemma 4.7 in the same paper, it is also dense in $C^\infty$ topology. It also contains the set of constant curvature metrics. We conjecture that it is actually the whole set of metrics of negative curvature.

\begin{lemma}\label{lemma:eta-1}
Let $g\in \mathfrak{G}(M)$. Let $\eta^0(g)= \sup\{ \Re z\ |\ L_0(z)=0 \}$. Then $\delta(g)\leq \eta^0(g) < \infty$. For any $\epsilon>0$, there is $\eta^0(g) < \eta < \eta^0(g) +\epsilon$ such that $\varphi$ does not vanish on the line $\{ \Re s = \eta \}$. Additionally, there is a $\eta(g) \geq \eta^0(g)$ such that for all $\eta > \eta(g)$, $|L_0|$ has a positive lower bound on the line $\{ \Re s = \eta\}$.
\end{lemma}

\begin{proof}
Since $\varphi$ is a non-zero meromorphic function, its zeroes are discrete and the lemma follows. Hence it suffices to consider the result concerning $L_0$. We decompose it as:
\[
L_0(s) = a^0_\ast e^{-s \ell_\ast}\left( 1 +  \sum_{n\geq 1}\frac{a^0_{\ell_\ast + n}}{a^0_\ast} e^{-s(\ell_{n^\ast + n} - \ell_\ast)} \right).
\]
When $\Re s$ is large enough, the RHS cannot vanish and has actually modulus $> 1/2$.
\end{proof}

This lemma generalizes case III of the main theorem in \cite{Bonthonneau-3}:
\begin{lemma}\label{lemma:resonance-free-zone}
Let $g\in \mathfrak{G}(M)$. There is a constant $C>0$ such that for all $\epsilon >0$ there is $C'>0$ such that $\varphi$ has no zeroes in the region
\[
\left\{z\in \C\ |\ \beta > \eta(g) + \epsilon,\ |\gamma| > C' + e^{C \beta} \right\}.
\]
\end{lemma}

\begin{proof}
In such a region, if $\epsilon'$ is small enough, since the lower bound on $|L_0(\eta + i\lambda)|$, $\lambda\in \R$ depends continuously on $\eta$,
\[
|L_0(s)| > \left| \sum_{j\geq 1} s^{-j} L_j(s) \right|.
\]
Hence a direct application of Rouch\'e's theorem leads to the conclusion.
\end{proof}

Let us introduce $\mathscr{D}_\eta$ the set of Dirichlet series whose abscissa of absolute convergence is strictly smaller than $\eta$, and are bounded for $\Re s > \eta$ --- i.e $\lambda_0 \geq 1$. Also consider $\mathscr{D}_\eta^0$ those in $\mathscr{D}_\eta$ that tend to zero as $\Re s \to \infty$ (i.e $\lambda_0 >1$).

Throughout the rest of this section, we will consider a fixed $\eta>\eta(g)$ taken according to lemma \ref{lemma:eta-1} so that $\varphi$ does not vanish on $\{ \Re s = \eta\}$. In that case, when $\Re s = \eta$, we have
\[
\frac{\varphi'}{\varphi} = \frac{L_0'}{L_0} +\mathcal{O}(1/s),
\]
where the remainder $\mathcal{O}(1/s)$ has an expansion similar to \eqref{eq:parametrix-phi}. We can expand this further and find $\tilde{L}_0\in \mathscr{D}_\eta^0$ such that
\begin{equation}\label{eq:P1}
\frac{\varphi'}{\varphi}=  - \ell_\ast + \tilde{L}_0 + \mathcal{O}\left( \frac{1}{s} \right).
\end{equation}
With a similar reasoning, we can find $\tilde{L}_1$ in $\mathscr{D}_\eta^0$, such that for $\Re s = \eta$
\begin{equation}\label{eq:P2}
\log | \varphi(s)| = - \frac{\kappa d}{2}\log |s| - \eta\ell_\ast + \log |a^0_\ast| + \Re \tilde{L}_1 + \mathcal{O}\left( \frac{1}{s} \right).
\end{equation}

These properties will be used repeatedly in the proof.

\subsection{Some lemmas from harmonic analysis}

Let us start with some abstract lemmas on zeros of holomorphic functions. Take $F$ a function holomorphic in a neighbourhood of a half plane $\{\Re z \geq a\}$. All sums are over the zeros of $F$, denoted by $z=\beta + i\gamma$ --- following Selberg's notations in \cite{Selberg-2}. When a zero is sitting on the smooth boundary of the counting box, it is counted with half multiplicity (and quarter multiplicity on right corners)
\begin{lemma}[Carleman]\label{lemma:Carleman}
Assume $\eta>a$, and $\lambda>0$, and assume that $F$ does not vanish on $\{ \Re z = \eta \}$. Then
\begin{equation*}
\begin{split}
2\pi \sum_{\beta > \eta,\ |z-\eta| < \lambda} \log \frac{\lambda}{|z-\eta|} = \int_{-\pi/2}^{\pi/2} &\log | F(\eta + \lambda e^{i \theta})| d \theta - \pi\log | F(\eta)|\\
					& {} + \int_{-\lambda}^\lambda \log \frac{\lambda}{|t|} \Re \frac{F'(\eta+it)}{F(\eta+it)} dt.
\end{split}
\end{equation*}
\end{lemma}

Now, additionally assume that $a=d/2$, that $|F|=1$ on the axis $\{ \Re s = d/2\}$, and that $F$ is real on the real axis.
\begin{lemma}[Counting in big rectangles] 
For $\lambda>0$,
\begin{equation}\label{eq:counting_big_box}
\begin{split}
2\pi \hspace{-0.4cm}\sum_{\substack{d/2 \leq \beta \leq \eta,\\ 0 \leq \gamma \leq \lambda}} \hspace{-0.4cm} (\lambda- \gamma)&(\beta - d/2) = \int_{d/2}^\eta  \log \frac{ |F(x + i\lambda)|}{|F(x)|} (x-d/2) dx  \\
				 {}+ \int_0^\lambda & (\lambda - t) \left[ \Re \frac{F'}{F}(\eta+it) (\eta-d/2) -  \log |F(\eta+it)| \right] dt.
\end{split}
\end{equation}
\end{lemma}

\begin{lemma}[Counting in small rectangles]
Take $c>0$. For $\lambda>0$,
\begin{equation}\label{eq:counting_small_box}
\begin{split}
2\pi \hspace{-0.4cm}&\sum_{\substack{|\gamma - \lambda| \leq \pi/2c \\ d/2 \leq \beta \leq \eta}} \hspace{-0.4cm}\cos(  c(\gamma- \lambda)) \sinh( c (\beta - d/2))  =  \\
								 &\int_{-\pi/2c}^{\pi/2c} \sinh\Big[c(\eta-\frac{d}{2})\Big] \cos( c t) \Re \frac{F'}{F}(\eta + i\lambda + it) dt \\
								{}- c &\int_{-\pi/2c}^{\pi/2c}  \cosh\Big[c(\eta-\frac{d}{2})\Big] \cos(c t) \log |F(\eta+i\lambda +it)| dt \\
								{}+ c &\int_{\frac{d}{2}}^\eta \log \left[ \Big|F(x +i\lambda + i\frac{\pi}{2c} )\Big|.\Big|F(x+ i\lambda- i \frac{\pi}{2c})\Big|\right] \sinh\Big[c(x-\frac{d}{2})\Big] dx.
\end{split}
\end{equation}
\end{lemma}

\begin{proof}
These three counting lemmas are obtained by considering the fact that $\log |F|$ is a harmonic function where $F$ does not vanish. Hence, if $u$ is another harmonic function on some open set $\Omega$, such that $F$ does not vanish on $\partial \Omega$, by Stoke's theorem ($\partial_\nu$ is the outward pointing normal derivative)
\begin{equation*}
2\pi \sum_{z \in \Omega, F(z)=0} u(z) = \int_{\partial \Omega} u \partial_{\nu} \log |F| - \log |F| \partial_\nu u.
\end{equation*}

If $F$ vanishes on the boundary of $\Omega$ (resp. on a right corner of the boundary), by removing small half (resp. quarter) disks around those zeros, one find that they are counted with multiplicity $1/2$ (resp. $1/4$), in a similar formula.

For \ref{lemma:Carleman}, we consider $u(z) = -\log |z-\eta|/\lambda$, and integrate on the boundary of the half-disk. For \eqref{eq:counting_big_box}, we take $u(z) = (\lambda- \Im z)(\Re z - 1/2)$, and finally $u(z) = \cos(c(\Im s- \lambda)) \sinh( c (\Re z - 1/2))$ for \eqref{eq:counting_small_box}.

\end{proof}
The estimates on counting in boxes are similar to equations (1.1) in \cite{Selberg-2}, and lemma 14, p. 319 in \cite{Selberg-1}. The Carleman lemma is reminiscent of the usual Carleman theorem \cite[\S 3.7]{Titchmarsh}. Last of this section is
\begin{lemma}\label{lemma:mean_value_Dirichlet_series}
Let $L \in \mathscr{D}^0_\eta$. Then, as $\lambda \to \infty$, 
\begin{equation*}
\int_0^\lambda \Re L(\eta + it) dt = \mathcal{O}(1).
\end{equation*}
\end{lemma}

\begin{proof}
Since $L$ converges absolutely in the region we are considering, we can write
\begin{equation*}
L(\eta+it) = \frac{c_0}{\lambda_0^{it}} + \frac{c_1}{\lambda_1^{it}} + \dots + \frac{c_k}{\lambda_k^{it}} + \dots
\end{equation*}
where the $c_k$'s are real, the $\lambda_k$'s are real, ordered, and strictly greater than $1$, and the sum converges normally. So we can estimate
\begin{equation*}
\int_0^\lambda \Re L = \sum_k c_k \frac{\sin (\lambda \log \lambda_k )}{\log \lambda_k}.
\end{equation*}
Since all $\lambda_k$'s are bigger than $\lambda_0 > 1$, and since $\sum |c_k| < \infty$, we conclude.
\end{proof}

\subsection{Two estimates}

Before proving the actual Weyl law in the strip, we give two remarkable bounds on other counting quantities. First, we deal with resonances that are not in the strip, and then we give an asymptotics for a weighted counting function. The latter is crucial for the proof of the Weyl estimate. We pick $\eta$ according to lemma \ref{lemma:eta-1}.

\subsubsection{Counting resonances far from the spectrum}

We apply lemma \ref{lemma:Carleman} to $\varphi$ --- recall that zeroes are written $z= \beta + i\gamma$:
\begin{equation*}
\begin{split}
2\pi \hspace{-0.4cm} \sum_{\beta > \eta,\ |z-\eta| < \lambda} \hspace{-0.4cm} \log &\frac{\lambda}{|z-\eta|} = \\
\int_{-\lambda}^\lambda & \log \frac{\lambda}{|t|} \Re \frac{\varphi'(\eta+it)}{\varphi(\eta+it)} dt +\int_{-\pi/2}^{\pi/2} \log | \varphi(\eta + \lambda e^{i \theta})| d \theta - \pi\log | \varphi(\eta)|.
\end{split}
\end{equation*}
Since in the LHS the terms are all positive, we only need to find an upper bound for the RHS. A consequence of \eqref{eq:P1} is that  $\Re \varphi'/\varphi$ is bounded on $\{ \Re s = \eta\}$. Hence the first term is $\mathcal{O}(\lambda)$. The last term there is a constant. For the second term, we can use the estimate in lemma \ref{lemma:Maass-Selberg-estimate2}. This gives an upper bound of the form $\mathcal{O}(\lambda) + I$ where
\begin{equation*}
I = \int_{0}^{\pi/2} F\left(\frac{ \frac{2\eta -d}{2\lambda} + \cos \theta}{\sin \theta}\right) d\theta.
\end{equation*}
Here $F$ is smooth increasing function with $F(0)=0$ and such that $F(x) \sim \kappa \log x$ when $x\to \infty$. Since the integrand is integrable, positive and the resulting integral a decreasing function of $\lambda$, it is $\mathcal{O}(1)$. We have proved estimate \eqref{eq:counting-out-strip-global}, that is:
\begin{equation*}
\#\left\{ z \text{ zero }|\ \beta > \eta,\ |z| \leq \lambda \right\} = \mathcal{O}(\lambda)
\end{equation*}

Still using the Carleman estimate, we seek to prove estimate \eqref{eq:counting-out-strip-local}. Let $0<\tilde{\lambda} < \lambda/2$. Then we have
\begin{equation*}
\begin{split}
2\pi \hspace{-0.4cm} \sum_{\beta > \eta,\ |z-(\eta+i\lambda)| < \tilde{\lambda}} \hspace{-0.4cm} \log &\frac{ \tilde{\lambda}}{|z-(\eta+i\lambda)|} 	= \int_{-\tilde{\lambda}}^{\tilde{\lambda}} \log \frac{\tilde{\lambda}}{|t|} \Re \frac{\varphi'}{\varphi}(\eta+i(\lambda+t)) dt \\
										&+\int_{-\pi/2}^{\pi/2} \log | \varphi(\eta + i\lambda +  \tilde{\lambda} e^{i \theta})| d \theta - \pi\log | \varphi(\eta + i\lambda)|.
\end{split}
\end{equation*}
According to property \eqref{eq:P1}, the first term is $\mathcal{O}(\tilde{\lambda})$ --- independently of $\lambda$. Using property \eqref{eq:P2}, we see that the last term is $\mathcal{O}(\log \lambda)$. According to lemma \ref{lemma:Maass-Selberg-estimate2}, we find that the second term is less than
\begin{equation*}
\mathcal{O}(1+ \tilde{\lambda}) +  \int_{-\pi/2}^{\pi/2} F\left(\frac{ \eta - \frac{d}{2} + \tilde{\lambda} \cos \theta}{|\lambda + \tilde{\lambda} \sin \theta|}\right) d\theta. 
\end{equation*}
Since $\tilde{\lambda} < \lambda /2$, the integrand is uniformly bounded, and we deduce that 
\begin{equation}\label{eq:general-counting-out-of-strip}
\#\left\{\rho\ |\ \Re \rho >\eta,\ |\rho - \eta - i\lambda|\leq \tilde{\lambda}  \right\} = \mathcal{O}(\tilde{\lambda} + \log \lambda).
\end{equation}

To close this section, let us just remark that lemma \ref{lemma:resonance-free-zone} implies that if $\tilde{\lambda} = o(\log \lambda)$, the set we are counting is empty.

\subsubsection{Counting with weights in vertical strips.}

Now, we will prove estimate \eqref{eq:Mangoldt-estimate}. The first step is to prove a local estimate. To this effect, we use formula \eqref{eq:counting_small_box}. With $c=\pi/2$,
\begin{equation*}
\begin{split}
2\pi \hspace{-10pt} \sum_{\substack{\beta\leq \eta \\ |\gamma - \lambda| \leq  1/2}}& \hspace{-10pt} \frac{\pi}{2\sqrt{2}}(\beta - \frac{d}{2}) \leq \\
 								&\int_{-1}^{1} \sinh\Big[\frac{\pi}{2}(\eta-\frac{d}{2})\Big] \cos( \frac{\pi}{2} t) \Re \frac{\varphi'}{\varphi}(\eta + i\lambda + it) dt \\
								{}- \frac{\pi}{2} &\int_{-1}^{1}  \cosh\Big[\frac{\pi}{2}(\eta-\frac{d}{2})\Big] \cos(\frac{\pi}{2} t) \log |\varphi(\eta+i\lambda +it)| dt \\
								{}+ \frac{\pi}{2} &\int_{\frac{d}{2}}^\eta \log \left[ \Big|\varphi(x +i\lambda + i )\Big|.\Big|\varphi(x+ i\lambda- i )\Big|\right] \sinh\Big[\frac{\pi}{2}(x-\frac{d}{2})\Big] dx.
\end{split}
\end{equation*}
In the RHS, the first integral is $\mathcal{O}(1)$ according to estimate \eqref{eq:P1} and the second one is $\mathcal{O}(\log \lambda)$ thanks to \eqref{eq:P2}. For the last integral, we can use lemma \ref{lemma:Maass-Selberg-estimate2} to see that it is less than $\mathcal{O}(1)$. We deduce that
\begin{equation}\label{eq:weighted_counting_resonances_small_box}
\sum_{\substack{ \beta \leq \eta \\ \lambda \leq \gamma \leq \lambda+ 1}} \beta - \frac{d}{2} = \mathcal{O}(\log \lambda).
\end{equation}
The aim is to prove an integrated version of this estimate. Using lemma \ref{lemma:Maass-Selberg-estimate2} again, we see that there is a constant $C>0$ such that
\begin{equation}\label{eq:upper_bound_horizontal_lines}
\int_{d/2}^\eta \log |\varphi(x + i\lambda)| (x-\frac{d}{2}) dx \leq C.
\end{equation}
We will also need a lower bound. Let us proceed as in Selberg \cite[p.21]{Selberg-2}. The LHS in \eqref{eq:counting_small_box} is always positive, so we write
\begin{equation*}
\begin{split}
c \int_{d/2}^\eta \log \Big(\Big|\varphi(x +i\lambda &+ i\frac{\pi}{2c} )\Big|.\Big|\varphi(x+ i\lambda- i \frac{\pi}{2c})\Big|\Big) \sinh\Big[c(x-\frac{d}{2})\Big] dx > \\
								&{} -\int_{-\frac{\pi}{2c}}^{\frac{\pi}{2c}} \sinh\Big[c(\eta-\frac{d}{2})\Big] \cos( c t) \Re \frac{\varphi'}{\varphi}(\eta + i\lambda + it) dt \\
								&{} + c \int_{-\frac{\pi}{2c}}^{\frac{\pi}{2c}}  \cosh\Big[c(\eta-\frac{d}{2})\Big] \cos(c t) \log |\varphi(\eta+i\lambda +it)| dt
\end{split}
\end{equation*}
Equation \eqref{eq:P1} implies that the first term in the RHS is $\mathcal{O}(1)$ ($\Re \varphi'/\varphi$ is bounded on $\Re s = \eta$). Then, \eqref{eq:P2} implies that the second term is $\mathcal{O}(\log \lambda)$ (since $\log |\varphi| = \mathcal{O}(\log \lambda)$ on $\Re s = \eta$). Using the upper bound \eqref{eq:upper_bound_horizontal_lines} on $\log | \varphi(z)|$ given by Maass-Selberg, we find for some constant $C>0$ depending on $c$,
\begin{equation*}
\int_{d/2}^\eta \log \Big|\varphi(x +i\lambda - i\frac{\pi}{2c} )\Big| \sinh\Big[c(x-\frac{d}{2})\Big] dx <  \mathcal{O}(1),
\end{equation*}
so that
\begin{equation*}
\int_{d/2}^\eta \log \Big|\varphi(x +i\lambda + i\frac{\pi}{2c} )\Big| \sinh\Big[c(x-\frac{d}{2})\Big] dx  > - C \log \lambda.
\end{equation*}
We want to replace the $\sinh$ by $(x-d/2)$, but $\log |\varphi|$ could oscillate. The trick is to use again lemma \ref{lemma:Maass-Selberg-estimate2}: for some $c'>0$, $|\varphi(x+i\lambda + i\pi/2c)| \leq 1/c'$ for all $d/2 < x < \eta$ and $\lambda >0$. So that for some constant $C>0$,
\begin{equation*}
\begin{split}
\int_{d/2}^\eta \log \left(c'\Big|\varphi(x +i \lambda + i\frac{\pi}{2c} )\Big|\right)& (x-\frac{d}{2}) dx \geq  \\
C  \int_{d/2}^\eta &\log \left(c'\Big|\varphi(x +i \lambda + i\frac{\pi}{2c} )\Big|\right) \sinh(c(x-\frac{d}{2})) dx.
\end{split}
\end{equation*}
We conclude that
\begin{equation}\label{eq:bound_horizontal_line}
\int_{d/2}^\eta \log |\varphi(x + i\lambda)| (x-\frac{d}{2}) dx = \mathcal{O}(\log \lambda).
\end{equation}

Now, we apply equation \eqref{eq:counting_big_box} (counting in big rectangles) for $\varphi$ at $\lambda$ and at $\lambda+1$, and we subtract the two equalities. We obtain
\begin{equation}\label{eq:exact_expression_counting_big_box_weighted_resonances}
\begin{split}
2\pi \sum_{\substack{d/2 \leq \beta \leq \eta \\ 0 \leq \gamma \leq \lambda}}  (\beta - \frac{d}{2}) =  &\int_0^\lambda \Re \left[ \frac{\varphi'}{\varphi}(\eta+it)\right] (\eta-\frac{d}{2})- \log |\varphi(\eta+it)|dt \\
				 &{}+ \int_{d/2}^\eta \log \frac{ |\varphi(x + i\lambda+i)|}{|\varphi(x + i \lambda)|} (x-\frac{d}{2}) dx \\
				 &{}+ R,
\end{split}
\end{equation}
where the remainder $R$ is
\begin{equation*}
\begin{split}
R = &- \sum_{\substack{d/2 \leq \beta \leq \eta\\ \lambda \leq \gamma \leq \lambda+1}}  (\lambda+1-\gamma) (\beta - \frac{d}{2}) \\
	&{}+ \int_0^1 \Re \left[\frac{\varphi'}{\varphi}(\eta+i(\lambda+1-t))\right] t(\eta-\frac{d}{2})dt \\
	&{}- \int_0^1 \log |\varphi(\eta+i(\lambda+1-t))| t dt.
\end{split}
\end{equation*}

In $R$, the first term is $\mathcal{O}(\log \lambda)$ by \eqref{eq:weighted_counting_resonances_small_box}. The second one is $\mathcal{O}(1)$ by \eqref{eq:P1}, and the last one is $C \log \lambda + \mathcal{O}(1)$ by \eqref{eq:P2}.

In the RHS of \eqref{eq:exact_expression_counting_big_box_weighted_resonances}, the second term is $\mathcal{O}(\log \lambda)$, as stated by equation \eqref{eq:bound_horizontal_line}. For the first term, we can use lemma \ref{lemma:mean_value_Dirichlet_series}, \eqref{eq:P1} and \eqref{eq:P2} to conclude that
\begin{equation}\label{eq:Mangoldt-estimate-bis}
2\pi \sum_{\substack{d/2 \leq \beta < \eta,\\ 0 \leq \gamma \leq \lambda}} \beta - \frac{d}{2} = \frac{\kappa d}{2} \lambda \log \lambda - \left( \frac{\kappa d}{2} + \log a^0_\ast - \frac{d}{2}\ell_\ast\right) \lambda+ \mathcal{O}(\log \lambda). 
\end{equation}
This was the desired estimate \ref{eq:Mangoldt-estimate} in theorem \ref{thm:Resonances-negative}.

\subsection{A Weyl law in the vertical strip}

To obtain a Weyl law for the resonances, we follow again Selberg's ideas. Let
\begin{equation}\label{eq:def-N}
N(\lambda):= \#\left\{ s\in \Res(M,g)\ \middle|\ d-\eta < \Re s \leq \frac{d}{2},\ 0\leq \Im s \leq \lambda \right\}
\end{equation}
($\eta>\eta(g)$ still fixed, given by lemma \ref{lemma:eta-1}). The first step is again a local bound
\begin{lemma}\label{lemma:estimate-log-box}
The number of zeroes in a rectangle with vertices $d/2 + i (\lambda \pm 1/\log \lambda)$ and $\eta + i(\lambda \pm 1/\log \lambda)$ is bounded by $\mathcal{O}(\lambda^d/\log \lambda)$.
\end{lemma}

\begin{proof}
First, we use formula \eqref{eq:decomp-increasing-phase}
\begin{equation*}
\begin{split}
\Psi(\lambda + 2/\log \lambda) - &\Psi(\lambda- 2/\log \lambda) = \mathcal{O}(\lambda^d/\log \lambda) \\
	&+ \frac{1}{2\pi}\sum_{z\text{ zero}} \int_{\lambda-2/\log \lambda}^{\lambda+ 2/\log \lambda} \frac{d- 2 \beta}{(\gamma - t)^2 + (\beta - d/2)^2} dt
\end{split}
\end{equation*}
In the RHS, all the terms in the sum are negative, so
\begin{equation*}
\begin{split}
\frac{1}{2\pi}\sum_{|z - d/2 + i \lambda| \leq 2/\log \lambda} &\int_{\lambda-2/\log \lambda}^{\lambda+ 2/\log \lambda} \frac{2 \beta - d}{(\gamma - t)^2 + (\beta - d/2)^2} dt \\
	&\leq \mathcal{O}(\lambda^d/\log \lambda) + \Psi(\lambda - 2/\log \lambda) - \Psi(\lambda + 2/\log \lambda).
\end{split}
\end{equation*}
However, when $|z - d/2 + i\lambda| \leq 2/\log \lambda$, 
\begin{equation*}
\int_{\lambda-2/\log \lambda}^{\lambda+ 2/\log \lambda} \frac{2\beta - d}{(\gamma - t)^2 + (\beta - d/2)^2} dt \geq \pi.
\end{equation*}
Also recall that $\widetilde{N}(\lambda) = N_{pp}(\lambda) - \Psi(\lambda)$, and $N_{pp}$ is non-decreasing. We deduce that the number of zeroes in the half ball of radius $2/\log \lambda$, centered at $d/2 + i\lambda$ is bounded by
\begin{equation*}
\widetilde{N}(\lambda+ 2/\log \lambda) - \widetilde{N}(\lambda - 2 / \log \lambda)  + \mathcal{O}(\lambda^d/\log \lambda) = \mathcal{O}(\lambda^d/\log \lambda),
\end{equation*}
according to estimate \eqref{eq:continuous-estimate-negative-K} in theorem \ref{thm:Phase}(the Weyl law for the phase). To finish the proof of our lemma, it suffices to prove that
\begin{equation*}
\# \left\{ z \ \middle| \ d/2 + \frac{\sqrt{3}}{2\log \lambda} < \beta < \eta ,\ |\gamma - \lambda| \leq 1/\log \lambda\right\} = \mathcal{O}(\lambda^d/\log \lambda).
\end{equation*}
However, according to \eqref{eq:weighted_counting_resonances_small_box} this quantity is bounded above by
\begin{equation*}
\frac{2 \log \lambda}{\sqrt{3}} \sum_{|\gamma - \lambda| \leq 1/\log \lambda} \beta - \frac{d}{2} = \mathcal{O}(\log^2 \lambda).
\end{equation*}
and this is a better bound than we need.
\end{proof}

Now, consider the rectangle $\mathcal{R}_\lambda$ with vertices $d/2$, $\eta$, $d/2 + i \lambda$ and $\eta+ i\lambda$. If we integrate $\varphi'/\varphi$ along its boundary, we obtain the number of zeroes $z=\beta + i \gamma$ of $\varphi$ in $\mathcal{R}_\lambda$. This gives
\begin{equation*}
\begin{split}
N(\lambda) :=N_{pp}(\lambda) + \sum_{\substack{\beta \leq \eta,\ 0\leq \gamma \leq \lambda\\ \varphi(z)=0}} 1 &= N_{pp}(\lambda) - \int_0^\lambda \Psi'(t)dt \\
				&+\frac{1}{2\pi} \int_0^\lambda \Re\left[ \frac{\varphi'}{\varphi}(\eta +i t) \right] dt \\
				& + \frac{1}{2\pi} \int_{d/2}^\eta \Im \left[ \frac{\varphi'}{\varphi}(\sigma) - \frac{\varphi'}{\varphi}(\sigma + i\lambda) \right] d\sigma.
\end{split}
\end{equation*}
According to estimate \eqref{eq:continuous-estimate-negative-K}, the first line in the RHS is 
\begin{equation*}
\frac{\vol(B^\ast M)}{(2\pi)^{d+1}} \lambda^{d+1} - \frac{\kappa \lambda}{\pi}\log \lambda + \frac{\kappa(1-\log 2)}{\pi}\lambda + \mathcal{O}(\lambda^d/\log \lambda).
\end{equation*}
Using \eqref{eq:P1}, one can see that the second term is $ \ell_\ast \lambda/(2\pi) + \mathcal{O}(\log \lambda)$. 

The last term is not so easy to estimate, but we have one more trick up our sleeve: the pigeon-hole principle. Thanks to the local bound given by lemma \ref{lemma:estimate-log-box}, it now suffices to find for each $\lambda$ a $\lambda'$ such that $\lambda- \lambda' = \mathcal{O}(1/\log \lambda)$, and such that the Weyl estimate holds for $\lambda'$. That is, to obtain the result we seek, we only have to prove
\begin{lemma}\label{lemma:pigeonhole}
For $\lambda$ large enough, there is a $\lambda'$ such that $\lambda-\lambda' = \mathcal{O}(1/\log \lambda)$, and
\begin{equation*}
J(\lambda'):=\int_{d/2}^\eta \Im\left[ \frac{\varphi'}{\varphi}(\sigma + i \lambda') \right] d\sigma = Q_2(\lambda') + \mathcal{O}\left( \lambda^{d/2} \log \lambda \right)= o(\lambda^d/\log \lambda).
\end{equation*}
where $Q_2$ is a polynomial of order at most $2\lfloor d/2 \rfloor - 1$.
\end{lemma}

This estimate shows that, to some extent, improving the remainder in the continuous Weyl law of section \ref{sec:weyl-phase} automatically improves the Weyl law for the resonances. This was pointed out by Selberg. 

\begin{proof}
Again, we use the decomposition \eqref{eq:decomp-increasing-phase} for $\varphi'/\varphi$ as a sum over the zeroes and a polynomial term. We will deal first with the polynomial. Then we will bound the contribution from zeroes out of the strip, and then we will come back to zeroes in the strip. 

The polynomial contribution to $J(\lambda)$ is
\begin{equation*}
\int_{d/2}^\eta \Re Q'(\sigma + i \lambda) d\sigma = \Re[ Q(\eta + i \lambda) - Q(d/2 + i \lambda) ].
\end{equation*}
However, since $Q$ is polynomial of order at most $2\lfloor d/2 \rfloor + 1$, and real for $\Re s = d/2$, the RHS here is a polynomial in real $\lambda$ of order at most $2\lfloor d/2 \rfloor - 1$. Let us call it $Q_2(\lambda)$.

Next, each zero $z$ contributes to the imaginary part of $\varphi'/\varphi$ by
\begin{equation*}
f_z(s):=\Im\left[ \frac{1}{s- z} - \frac{1}{s - d + \overline{z}} \right] = \frac{ (2\beta - d) (2 \Re s - d) (\gamma - \Im s)}{|s-z|^2 | s- d + \overline{z} |^2}.
\end{equation*}

\begin{lemma}
For $\lambda>1$,
\begin{equation*}
\int_{d/2}^\eta \Im \left[\frac{\varphi'}{\varphi}(\sigma + i\lambda)\right] d\sigma = \sum_{\substack{|d/2 + i\lambda -z|< 1,\\ \varphi(z)=0}} \int_{d/2}^\eta f_z(\sigma + i\lambda) d\sigma + Q_2(\lambda) + \mathcal{O}(\sqrt{\lambda}).
\end{equation*}
\end{lemma}

\begin{proof}
For the zeroes that are not in the strip $\{ d/2 \leq \Re s < \eta\}$, for $s \in i\lambda + [d/2,\eta]$, $|s-z| \approx |s- d +\overline{z}| \approx | z - d/2 - i\lambda|$, so we have the bound
\begin{equation*}
\int_{d/2}^\eta f_z(\sigma + i\lambda) d\sigma = \mathcal{O}(1) \frac{(2\beta - d)(\lambda - \gamma)}{|d/2 + i\lambda - z|^4}.
\end{equation*}
To estimate their total contribution, we have to distinguish three cases: either $| d/2 + i\lambda - z| > \lambda^{-1/6} |d/2 - z|$, or $| d/2 + i\lambda - z| < \lambda^{-1/6} |d/2 - z|$. In the latter case, either $|d/2 + i \lambda -z| \leq \sqrt{\lambda}$ or $|d/2 + i \lambda -z| > \sqrt{\lambda}$.

Assume that $| d/2 + i\lambda - z| > \lambda^{-1/6} |d/2 - z|$. Then the LHS is bounded by
\begin{equation*}
\mathcal{O}(1) \sqrt{\lambda} \frac{2 \beta - d}{|z - d/2|^3}= \sqrt{\lambda}\ \mathcal{O}\left( \frac{2 \beta - d}{|z - d/2|^2} \right).
\end{equation*}
Since $\sum_z (2\beta - d)/|z - d/2|^2 < \infty$, summing over such zeroes will contribute by $\mathcal{O}(\sqrt{\lambda})$ to $J(\lambda)$.

Now, we assume that $|d/2 + i\lambda -z| \leq \lambda^{-1/6} |z - d/2|$. For $\lambda$ large enough, by the triangular inequality, $|d/2 + i\lambda -z| \leq C \lambda^{5/6}$, for some positive constant $C$. 

Each such zero in the half annulus $\{ \sqrt{\lambda} \leq |d/2 + i \lambda -z| \leq C\lambda^{5/6},\ \beta > d/2\}$ contributes by
\begin{equation*}
\mathcal{O}(1)\frac{(2\beta - d)(\lambda - \gamma)}{|d/2 + i\lambda - z|^4} = \mathcal{O}(\frac{\lambda^{5/6} \lambda^{5/6} }{\sqrt{\lambda}^4})= \mathcal{O}(\lambda^{-1/3}).
\end{equation*}
and according to the estimate \eqref{eq:general-counting-out-of-strip}, there are at most $\mathcal{O}(\lambda^{5/6})$ such zeroes, so their total contribution is $\mathcal{O}(\lambda^{1/2})$.

From \eqref{eq:general-counting-out-of-strip}, we also deduce that there are at most $\mathcal{O}(\sqrt{\lambda})$ zeroes such that $|d/2 + i \lambda - z| \leq \sqrt{\lambda}$ outside of the strip. Since we have a logarithmic zone without resonances --- lemma \ref{lemma:resonance-free-zone} --- we deduce that their contribution to $J(\lambda)$ is $\mathcal{O}(\sqrt{\lambda} / \log^2\lambda)$. Hence, the total contribution of resonances out of the strip is $\mathcal{O}(\sqrt{\lambda})$.

Now, we turn to resonances in the strip. Here, we follow Selberg’s argument closely. Let $H>1$, and assume $H < |\lambda - \gamma| < 2 H$. Then $f_z = H^{-3} \mathcal{O}(2\beta - d)$ and
\begin{equation*}
\sum_{H < |\lambda - \gamma| < 2 H} \hspace{-70pt} \hspace{60pt}f_z(\sigma + i \lambda) = \frac{\mathcal{O}(1)}{H^3}\sum_{\substack{|\lambda-\gamma| < 2 H\\ \beta < \eta}} (2\beta - d)
\end{equation*}
Using the ``Riemann-Von Mangoldt'' formula \eqref{eq:Mangoldt-estimate-bis}, the sum in the RHS is 
\begin{equation*}
\mathcal{O}(1)\left[ (\lambda + 2 H) \log|\lambda + 2H| - (\lambda - 2H) \log |\lambda - 2H|+ H + \log (\lambda + 2H) \right]
\end{equation*}
In the case $2H \geq \lambda$, use the concavity of the logarithm to find that it is $\mathcal{O}( H \log H)$. When $\lambda > 2 H$ we see from elementary manipulations that it is $\mathcal{O}(H \log \lambda)$. Summing this estimate for $H = 2^n$, $n\geq 0$ proves that
\begin{equation*}
\sum_{\substack{|\lambda - \gamma| > 1,\\ \beta < b}} f_z(\sigma + i \lambda) = \mathcal{O}(\log \lambda).
\end{equation*}
\end{proof}

Now comes the delicate part. Let $\ell(\lambda)$ be a monotonic integer valued function such that $\ell(\lambda) \to +\infty$, and $h(\lambda)$ also monotonic, such that $h(\lambda) \to 0$. Also assume that $h^\ast(\lambda) =  \ell h(\lambda) = \mathcal{O}(1/\log \lambda)$. Then, from lemma \ref{lemma:estimate-log-box},
\begin{equation*}
\#\{ z,\ \beta \leq b,\ \lambda \leq \gamma \leq \lambda+h^\ast \} = \mathcal{O}(\lambda^d /\log \lambda)
\end{equation*}
By the pigeonhole principle, there is some integer $0 \leq v \leq \ell-1$ so that there are at most $\mathcal{O}(\lambda^d/ \log \lambda)/\ell$ resonances in the strip with $ \lambda + vh \leq \Im \rho \leq \lambda+ (v+1)h $. We let $\lambda' = \lambda + (v+1/2)h$.

Consider $\rho$ in the strip such that $|\gamma - \lambda'| > h$. Then 
\begin{equation*}
\begin{split}
\int_{d/2}^\eta f_z(\sigma + &i \lambda') d\sigma = \int_0^{\eta-d/2} \frac{ 2(2\beta -d)(\gamma - \lambda')\sigma d\sigma}{|z - \sigma -d/2 -i\lambda'|^2 |\sigma + d/2 + i\lambda' - d + \overline{z} |^2}\\
				&= \int_0^{(\eta-d/2)/(\gamma - \lambda')} \frac{ 2(2\beta -d)(\gamma - \lambda')^{-1}\sigma d\sigma}{(1+ (\sigma + \frac{d/2 - \beta}{\gamma - \lambda'})^2)(1 + (\sigma - \frac{d/2-\beta}{\gamma -\lambda'})^2)}\\
				&\leq \int_0^{(\eta-d/2)/(\gamma - \lambda')}\hspace{-5pt} 2\ \frac{ (2\beta -d)(\gamma - \lambda')^{-1} \sigma d\sigma}{(1+ \sigma^2)}\\
				&\leq \mathcal{O}(1)(2\beta - d) (\gamma - \lambda')^{-1} (1 + \log |\gamma - \lambda'|).
\end{split}
\end{equation*}
Observe how this is $(2 \beta - d)\mathcal{O}(h^{-1}|\log h|)$. If we sum this and use \eqref{eq:weighted_counting_resonances_small_box}, we get
\begin{equation*}
\sum_{\substack{h<|\gamma - \lambda'|< 1,\\ \beta < \eta}} \int_{d/2}^\eta f_z(\sigma + i\lambda') d\sigma = \frac{|\log h|}{h}\sum_{\substack{|\gamma - \lambda'| \leq 1,\\ \beta < \eta}} 2 \beta - d = \mathcal{O}(\frac{\log \lambda |\log h|}{h}).
\end{equation*}

The remaining terms are those with $|\gamma - \lambda| < h$. For them, the best bound we can give for their individual contribution is $\mathcal{O}(1)$. Indeed, their contribution is the variation of the argument of $(s-d+ \overline{z})/(s-z)$ as $s$ goes from $d/2 + i\lambda'$ to $\eta + i\lambda'$ along a horizontal line. This variation is at most $\pi$. Hence
\begin{equation*}
\sum_{|\gamma - \lambda| < h} \int_{d/2}^\eta f_z(\sigma + i\lambda')d\sigma = \frac{1}{\ell}\mathcal{O}(\lambda^d/\log \lambda).
\end{equation*}

Combining all the above,
\begin{equation*}
 \int_{d/2}^\eta \Im\left[\frac{\varphi'}{\varphi}(\sigma + i \lambda')\right] d\sigma = Q_2(\lambda) + \mathcal{O}\left( \sqrt{\lambda} + \frac{\log \lambda |\log h|}{h} + \frac{\lambda^d}{\ell \log \lambda} \right).
\end{equation*}
where $Q_2$ is a polynomial of order at most $2\lfloor d/2 \rfloor -1$. Now, we have to choose $\ell$ and $h$ so that $\ell\to +\infty$, $h\to 0$ and $\ell h = \mathcal{O}(1/\log \lambda)$. Consider $\ell(\lambda)=\lfloor \lambda^{d/2}\log^{-2}\lambda\rfloor$ and $h(\lambda)=\lambda^{-d/2}\log \lambda$. Then we have
\begin{equation*}
\frac{\log \lambda |\log h|}{h} + \frac{\lambda^d}{\ell \log \lambda} = \mathcal{O}(\lambda^{d/2} \log \lambda).
\end{equation*}

\end{proof}

\section{Counting with weaker assumptions}

Now, we will explain shortly how to obtain the theorem \ref{thm:Resonances-general}. The argument is very similar to the one in \cite{Bonthonneau-1}, itself based on the computations p282 in \cite{Muller-92}. Note that we are still counting \emph{zeroes} $z=\beta + i\gamma$.

The first observation is that since
\begin{equation*}
\sum_{\varphi(z)=0} \frac{2\beta - d}{|z - d/2|^2} < \infty,
\end{equation*}
we directly have 
\begin{equation*}
\sum_{|z - d/2| \leq \lambda} 2\beta - d = o(\lambda^2).
\end{equation*}
Hence, for $\epsilon >0$.
\begin{equation}\label{eq:bound-far-spectrum-general}
\#\{ z \text{ zero }|\ \beta \geq d/2 + \epsilon,\ |z - d/2| \leq \lambda\} = o\left(\frac{\lambda^2}{\epsilon}\right).
\end{equation}
Uniformly in $\epsilon>0$ as $\lambda \to + \infty$. Since we work in dimension $>2$, we only have to count the zeroes in a strip $\{ d/2 < \beta \leq d/2 + \epsilon \}$.

\subsection{Without any assumption}

According to formula \eqref{eq:decomp-increasing-phase}, together with \eqref{eq:continuous-estimate-general-K}, for $\lambda>0$,
\begin{equation}\label{eq:counting-general-1}
\begin{split}
N_{pp}(\lambda+1) - N_{pp}(\lambda) &+ \frac{1}{\pi}\sum_{\varphi(z)=0} \int_{\lambda}^{\lambda+1} \frac{ 2 \beta - d}{|z - d/2 + it|^2}dt \\
&= \mathcal{O}(\lambda^d).
\end{split}
\end{equation}
But each zero contributes by a positive term in the LHS, so
\begin{equation*}
\sum_{\substack{\beta < d/2 + \epsilon,\\ 1/3 \leq \gamma - \lambda \leq 2/3}} \int_{\lambda}^{\lambda+1} \frac{2 \beta - d}{|z - d/2 + it|^2}dt = \mathcal{O}(\lambda^d).
\end{equation*}
Each of the terms in the LHS is larger than $4\arctan(1/3\epsilon)$. Taking $\epsilon$ small enough, we find the local part of the theorem:
\begin{equation}\label{eq:counting-resonance-local-proof}
\#\{ s \in \Res(M,g)\ |\ |s - d/2 - i\lambda| =\mathcal{O}(1)\} = \mathcal{O}(\lambda^d).
\end{equation}

Now, the global version of \eqref{eq:counting-general-1} is
\begin{equation}\label{eq:counting-general-2}
\begin{split}
\#\{ r_i\ |\ d^2/4 + r_i^2 \text{ eigenvalue, } |r_i|\leq \lambda \} &+ \frac{1}{2\pi} \sum_{\varphi(z)=0} \int_{-\lambda}^{\lambda} \frac{ 2 \beta - d}{|z - d/2 + it|^2}dt \\
&= \frac{\vol(B^\ast M)}{(2\pi)^{d+1}} \lambda^{d+1} + \mathcal{O}(\lambda^d).
\end{split}
\end{equation} 
But, as in the equation 4.9 in \cite{Muller-92}, we find
\begin{equation*}
\begin{split}
\int_{-\lambda}^\lambda \frac{2 \beta - d}{|z - d/2 + it|^2}dt &= 2 \arctan \left[\frac{(2\beta - d) \lambda}{|z - d/2|^2}\left(1-\frac{\lambda^2}{|z - d/2|^2}\right)^{-1}\right] \\
			&+ \begin{cases} 0 & \text{ if }|z- d/2|>\lambda \\ 2\pi & \text{ else} \end{cases}.
\end{split}
\end{equation*}
We can rewrite \eqref{eq:counting-general-2} as
\begin{equation*}
\#\{ s \in \Res(M,g)\ |\ |s-d/2| \leq \lambda\} = 2\frac{\vol(B^\ast M)}{(2\pi)^{d+1}} \lambda^{d+1} + \mathcal{O}(\lambda^d) + R(\lambda),
\end{equation*}
where
\begin{equation*}
R(\lambda) = \frac{1}{\pi}\sum_{\varphi(z)=0}  \arctan \frac{(2\beta-d) \lambda}{|z - d/2|^2}\left(1-\frac{\lambda^2}{|z - d/2|^2}\right)^{-1}.
\end{equation*}
In this sum, the zeroes with $||z-d/2| - \lambda |> 1$ contribute by $\mathcal{O}(\lambda^2)$. Indeed, if $f(x):= |1-x^{-2}|^{-1}$, $f(x) \leq 1 + 1/|x-1|$ for $x\geq 0$. In particular, we find that
\begin{equation*}
\left| 1 - \frac{\lambda^2}{|z-d/2|^2} \right|^{-1} \leq 1 + \lambda.
\end{equation*}
Hence their total contribution is (recall $|\arctan(x)|\leq |x|$)
\begin{equation*}
\mathcal{O}(1)\sum_{\varphi(z)=0} \frac{2 \beta - d}{|z - d/2 |^2} \lambda^2 = \mathcal{O}(\lambda^2).
\end{equation*}
From \eqref{eq:bound_horizontal_line} and \eqref{eq:counting-resonance-local-proof}, we know that there are at most $\mathcal{O}(\lambda^d)$ zeroes with $||z - d/2| - \lambda| \leq 1$, and so we conclude that $R(\lambda) = \mathcal{O}(\lambda^d)$.

\subsection{The aperiodic case}

When the manifold is aperiodic, we can follow the same scheme of proof as above, just tightening the estimates to obtain the $o(\lambda^d)$ result. Let $0<\epsilon <1$. We have the more precise version of \eqref{eq:counting-general-1} --- using \eqref{eq:continuous-estimate-periodic-zero}
\begin{equation*}
\begin{split}
N_{pp}(\lambda+ \epsilon) - N_{pp}(\lambda - \epsilon) &+ \frac{1}{\pi}\sum_{\varphi(z)=0} \int_{\lambda-\epsilon}^{\lambda+\epsilon} \frac{ (2 \beta - d)dt}{|z - d/2 + it|^2} \\
&= C \epsilon \lambda^d + o(\lambda^d).
\end{split}
\end{equation*}
From there we deduce the local part of the theorem
\begin{equation*}
\# \{ s \in \Res(M,g)\ |\ |s - d/2 - i\lambda| \leq \epsilon \} = \mathcal{O}\big[(\epsilon +o(1)) \lambda^d\big].
\end{equation*}
Now, we combine the computations of M\"uller recalled above with estimate \eqref{eq:continuous-estimate-periodic-zero}:
\begin{equation*}
\{ s \in \Res(M,g)\ |\ |s - d/2| \leq \lambda \} = 2\frac{\vol(B^\ast M)}{(2\pi)^{d+1}} \lambda^{d+1} + o(\lambda^d) + R(\lambda)
\end{equation*}
with the same expression for $R(\lambda)$:
\begin{equation*}
R(\lambda) = \frac{1}{\pi}\sum_{\varphi(z)=0}  \arctan \frac{(2\beta - d) \lambda}{|z - d/2|^2}\left(1-\frac{\lambda^2}{|z - d/2|^2}\right)^{-1}.
\end{equation*}
Instead of splitting this sum into two parts as above, here, we need to split it into three parts. First, observe that when $||z - d/2| - \lambda| > 1/\epsilon$,
\begin{equation*}
\left | 1 - \frac{\lambda^2}{|z - d/2|^2} \right|^{-1} \leq 1 + \epsilon \lambda.
\end{equation*}
Hence the total contribution of such zeroes is $\mathcal{O}(\lambda + \epsilon \lambda^2)$. For the zeroes such that $||z - d/2| - \lambda | \leq \epsilon$, we can use the local bound for those close to the axis, and estimate \eqref{eq:bound-far-spectrum-general}. We obtain the following bound on their contribution: $\mathcal{O}((\epsilon + o(1)) \lambda^d) + o(\lambda^2/\epsilon)$ uniformly in $\epsilon$ as $\lambda \to +\infty$. Now, we turn to the case $\epsilon < ||z - d/2| - \lambda | \leq 1/\epsilon$. Their contribution is bounded by
\begin{equation*}
\left(\lambda + \frac{\lambda^2}{\epsilon}\right) \sum_{||z - d/2| - \lambda| \leq 1/\epsilon} \frac{2\beta - d}{|z - d/2|^2}.
\end{equation*}
This is $(\lambda + \lambda^2/\epsilon) \times o(1)$ as long as $\lambda\to + \infty$ and $\lambda \epsilon \to +\infty$. We deduce that as $\lambda \to +\infty$ and assuming $\lambda \epsilon \to + \infty$,
\begin{equation*}
R(\lambda) = \mathcal{O}(\epsilon \lambda^d) + o(1)\left(\lambda^d +  \frac{\lambda^2}{\epsilon}\right).
\end{equation*}
This implies that $R(\lambda) = o(\lambda^d)$.

\appendix

\section{Lemmas for the scattering determinant}

In this section, we gather some lemmas of independent interest on the spectral theory of manifolds with cusp. They are not new, but as far as we know, the statements in the literature are not given in this level of generality.

\begin{lemma}[Maass-Selberg]\label{lemma:Maass-Selberg-estimate2}
The scattering determinant satisfies
\begin{equation*}
| \varphi( \sigma + it )| \leq y_0^{\kappa (2\sigma - d)} \left( \sqrt{ 1 + \frac{ (\sigma-d/2)^2}{t^2}} + \frac{ \sigma-d/2}{|t|}\right)^{\kappa} \text{ when } \sigma \geq d/2.
\end{equation*}
In particular, $|\varphi(s)|$ is bounded in the vertical strip $\{ d/2 < \Re s < b\}$, away from eventual poles in $[d/2, d]$.
\end{lemma}

One can find a proof for this statement in the footnote p.22 in \cite{Selberg-2} and the discussion before equation (7.44) p.652 in \cite{Selberg-1} in the case of constant curvature surfaces. We just check here that is also valid in the general case.

\begin{proof}
To prove this, we use the Eisenstein series. For each cusp $Z_\ell$, recall that $E_\ell(s)$ is a meromorphic family of smooth functions on $M$ that satisfy
\begin{equation}\label{eq:eigenfunction}
-\Delta E_\ell(s) = s(d-s) E_\ell(s).
\end{equation}
Additionally, the zeroth Fourier coefficient of $E_\ell(s)$ in cusp $Z_j$ equals
\begin{equation*}
f_{\ell j}(y,s)= \delta_{\ell j} y^s + \phi_{\ell j}(s) y^{d-s}.
\end{equation*}
We denote by $W(s,y)$ the matrix whose coefficients are the $f_{\ell j}$ --- $y > y_0$. Recall $\Pi_0$ is the $L^2$ orthogonal projector on functions supported in $\{y> y_0\}$ that do not depend on $\theta$. Then we define $G_\ell^\tau(s)= (\mathbb{1}-\mathbb{1}_{y>e^\tau}\Pi_0) E_\ell(s)$. This makes sense if $\tau \geq \log y_0$; this time $\tau$ will not tend to $+\infty$. Let $V(s,\tau)$ be the matrix with coefficients $\int_M G_\ell^\tau\overline{G_j^\tau}$. We set to prove the Maass-Selberg formula: for $\Re s > d/2$,
\begin{equation}\label{eq:Maass-Selberg-relation}
V(s,\tau) = \frac{ e^{2i \Im s \tau} \phi^\ast(s) - e^{-2i\Im s \tau} \phi(s)}{2 i \Im s} + \frac{e^{(2\Re s - d)\tau} - e^{(d-2\Re s)\tau} \phi \phi^\ast(s)}{2\Re s - d}.
\end{equation}
The trick is to compute $\partial_s [2i \Im (s(d-s)) \int_M G_\ell^\tau\overline{G_j^\tau}]$, i.e
\begin{equation*}
(d-2s)\int_M G_\ell^\tau(s)\overline{G_j^\tau(s)} + 2i \Im(s(d-s)) \int_M \partial_s G_\ell^\tau \overline{G_j^\tau(s)}.
\end{equation*}
($\overline{G_j^\tau(s)}$ is anti-holomorphic). Differentiating \eqref{eq:eigenfunction} with respect to $s$, we find this is equal to 
\begin{equation*}
\int_M - \overline{G_j^\tau(s)} \Delta \partial_s G_\ell^\tau   - \overline{s(d-s)} \int_M \partial_s G_\ell^\tau \overline{G_j^\tau}.
\end{equation*}
Now, we can use Stoke's formula twice to have the Laplacian hit $\overline{G_j^\tau}$ instead of $G_\ell^\tau$. The result is a ``boundary term'' plus a cancelling integral over $M$. The boundary term is
\begin{equation*}
\int_M \overline{E_j} \nabla( \mathbb{1} - \mathbb{1}_{y>e^\tau} \Pi_0)\cdot \nabla \partial_s E_\ell - \int_M \partial_s E_\ell \nabla( \mathbb{1} - \mathbb{1}_{y>e^\tau} \Pi_0)\cdot \nabla  \overline{E_j}.
\end{equation*}
This integral only involves zero Fourier modes in the cusps. Since the lattices $\Lambda_{k}$ have co-volume $1$, this is equal to  
\begin{equation*}
\sum_{k }\left[\partial_s f_{\ell k}\overline{\partial_y f_{jk}} - \overline{f_{jk}}\partial_y \partial_s f_{\ell k}  \right](s, e^\tau)
\end{equation*}
This is the $(\ell j)$ coefficient of the matrix
\begin{equation*}
\partial_s W \partial_y  W^\ast -  \partial_s \partial_y W. W^\ast = \partial_s( W.\partial_y W^\ast - \partial_y W. W^\ast).
\end{equation*}
We deduce that there is a anti-meromorphic matrix-valued function $A(s,\tau)$ such that
\begin{equation*}
2i \Im (s(d-s)) V(s,\tau) = A(s,\tau) + (W.\partial_y W^\ast - \partial_y W. W^\ast)(s,e^\tau).
\end{equation*}
After expanding the RHS with the expression for the $f_{\ell j}$, this is equivalent to
\begin{equation*}
\begin{split}
 2i (d-2\Re& s)\Im s V(s,\tau) - A(s,\tau) = \\
&(2\Re s -d)(e^{-2i\Im s\tau} \phi - e^{2i \Im s\tau} \phi^\ast) + 2 i \Im s ( e^{(d-2\Re s)\tau} \phi \phi^\ast - e^{(2\Re s - d)\tau}).
\end{split}
\end{equation*}
We deduce that $A(s)$ vanishes on the unitary axis $2\Re s =d$, and thus has to vanish identically, and the proof of \eqref{eq:Maass-Selberg-relation} is complete.

The matrix $V$ on the LHS of \eqref{eq:Maass-Selberg-relation} is non-negative, so that as a hermitian quadratic form,
\begin{equation*}
\phi \phi^\ast  \leq e^{2\tau(2\Re s - d)} + \frac{ 2 \Re s - d}{\Im s} \frac{e^{2s \tau} \phi^\ast - e^{2\overline{s}\tau} \phi}{2i e^{\tau d}}.
\end{equation*}
We deduce that
\begin{equation*}
\phi \phi^\ast \leq e^{2\tau(2\Re s - d)} \left( \sqrt{1 + \left( \frac{\Re s - d/2}{\Im s} \right)^2} + \frac{\Re s - d/2}{|\Im s|} \right)^2.
\end{equation*}
This formula is true as long as $e^\tau \geq y_0$. 
\end{proof}

Observe that taking the limit $\Re s \to d/2$ in \eqref{eq:Maass-Selberg-relation}, we find
\begin{equation*}
V\left(\frac{d}{2} + i \lambda  ,\tau\right) = 2 \tau \mathbb{1} + \frac{e^{2i \lambda \tau} \phi^\ast - e^{-2 i \lambda \tau} \phi}{2 i \lambda} - \frac{1}{2}(\phi' \phi^\ast + \phi {\phi^\ast}').
\end{equation*}
Since $\phi$ is unitary on the unitary axis, $\phi' \phi^\ast$ is self-adjoint, and we recover the classical form of the Maass-Selberg relations that was used in section \ref{sec:weyl-phase}:
\begin{equation}\label{eq:Maass-Selberg-axis}
\Tr V\left(\frac{d}{2} + i \lambda  ,\tau\right) = 2 \kappa \tau - \frac{\varphi'}{\varphi}\left(\frac{d}{2} + i \lambda  \right) + \Tr \frac{e^{2i \lambda \tau} \phi^\ast - e^{-2 i \lambda \tau} \phi}{2 i \lambda}.
\end{equation}

Now, we turn to the factorization of the scattering determinant. To readers accustomed to scattering theory in one dimension, the following lemma will not be a surprise
\begin{lemma}\label{lemma:factorization}
There is a polynomial $Q$ of order at most $2\lfloor d/2 \rfloor + 1$, such that
\begin{equation*}
\varphi(s) = \varphi\left( \frac{d}{2} \right) e^{i Q(s)} \prod_{z \text{ zero}} \frac{s- z}{s-d + \overline{z} }
\end{equation*}
Additionally, $Q$ is real for $\Re s = d/2$, and $Q(s) + Q(d-s)$ is constant. We also have
\begin{equation}\label{eq:convergence-somme}
\sum_{z \text{ zero}} \frac{ 2\Re z -d}{|z - d/2|^2} < \infty.
\end{equation}
\end{lemma}

This lemma is due to M\"uller (and Zworski, see theorem 3.31 in \cite{Muller-92}) in the case of surfaces (and Selberg for hyperbolic surfaces). Actually, the proof of M\"uller easily extends to the higher dimensional case --- one can find more explanations page 8, proposition 1.1.3 in \cite{Bonthonneau-thesis}. A consequence we used in section \ref{sec:weyl-phase} is the following:
\begin{equation}\label{eq:decomp-increasing-phase}
2\pi \Psi'(\lambda) = i Q'(d/2 + i\lambda) + \sum_{z \text{ zero}} \frac{d - 2 \Re z}{(d-\Re z /2)^2 + (\lambda - \Im z)^2}
\end{equation}
Observe how the second summand in the RHS is a negative function.

\end{document}